\documentclass{article}
\usepackage{eurosym}
\usepackage{amssymb}
\usepackage{amsfonts}
\usepackage{amsmath}
\newtheorem{theorem}{Theorem}

\newtheorem{corollary}[theorem]{Corollary}

\newtheorem{definition}[theorem]{Definition}

\newtheorem{lemma}[theorem]{Lemma}

\newtheorem{proposition}[theorem]{Proposition}
\newtheorem{remark}[theorem]{Remark}

\newenvironment{proof}[1][Proof]{\noindent\textbf{#1.} }{\ \rule{0.5em}{0.5em}}
\input{tcilatex}
\begin{document}

\begin{center}
{\Large Sensitivity Analysis with respect to a Stock Price Model
with Rough Volatility via a Bismut-Elworthy-Li Formula for Singular SDEs}

\bigskip

Emmanuel Coffie, Sindre Duedahl and Frank\ Proske

\bigskip

{\large Abstract}
\end{center}

In this paper, we show the existence of unique Malliavin differentiable
solutions to SDE`s driven by a fractional Brownian motion with Hurst
parameter $H<\frac{1}{2}$ and singular, unbounded drift vector fields, for
which we also prove a stability result. Further, using the latter results,
we propose a stock price model with rough and correlated volatility, which
also allows for capturing regime switching effects. Finally, we also derive
a Bismut-Elworthy-Li formula with respect to our stock price model for
certain classes of vector fields.

\bigskip \emph{keywords}: Bismut-Elworthy-Li formula, singular SDEs,
fractional Brownian motion, Malliavin calculus, stochastic flows, stochastic
volatility


\bigskip

\section{Introduction\protect\bigskip}

Consider the stochastic differential equation (SDE)%
\begin{equation}
dX_{t}^{x}=b(t,X_{t}^{x})dt+dB_{t}^{H},0\leq t\leq T,X_{0}^{x}=x\in \mathbb{R%
}^{d}\text{,}  \label{SDE0}
\end{equation}%
where $b:\left[ 0,T\right] \times \mathbb{R}^{d}\longrightarrow \mathbb{R}%
^{d}$ is a Borel-measurable function and $B_{t}^{H},0\leq t\leq T$ is a
fractional Brownian motion with Hurst parameter $H<\frac{1}{2}$.

It was shown in \cite{BNP} by using techniques from Malliavin calculus (see
e.g. \cite{Nualart}) that the SDE (\ref{SDE0}) admits the existence of a
unique (global) strong solution $X_{\cdot }^{x}$, when 
\begin{equation}
b\in L^{1}(\mathbb{R}^{d};L^{\infty }([0,T];\mathbb{R}^{d}))\cap L^{\infty }(%
\mathbb{R}^{d};L^{\infty }([0,T];\mathbb{R}^{d}))\text{.}  \label{bbounded}
\end{equation}%
Here, a solution to (\ref{SDE0}) is called strong, if it can be represented
as a (progressively) measurable functional of the driving noise $B_{\cdot
}^{H}$.

Related results in the direction of \cite{BNP},whose proofs, however, are
based on different methods, can be e.g. found in \cite{NO}, \cite{CG} and 
\cite{Le}.

By employing the Malliavin calculus approach developed in the fundamental
papers of \cite{F1} and \cite{F2} in the Wiener process case, the authors in 
\cite{ACHP} derive a so-called Bismut-Elworthy-Li formula (BEL-formula) for
(strong) solutions to (\ref{SDE0}), when the vector field $b$ is singular in
the sense of (\ref{bbounded}). Roughly speaking, such a formula, gives a
representation of expressions of the form%
\begin{equation*}
\frac{d}{dx}E\left[ \Phi (X_{T}^{x})\right]
\end{equation*}%
for functions $\Phi :\mathbb{R}^{d}\longrightarrow \mathbb{R}$, which
doesn`t involve the derivative of $\Phi $. In this context, we also mention
the interesting work \cite{FanRen}, where a BEL-formula is established for
(differentiable) functional drift coefficients $b$ and used to study Harnack
type of inequalities. See also the article \cite{MMNPZ} in the case of a
Wiener process.

In this paper, we aim at extending the result in \cite{ACHP} based on
Malliavin calculus to the case, when $b\in L^{1}(\mathbb{R}^{d};L^{\infty
}([0,T];\mathbb{R}^{d}))$ (see Theorem \ref{StrongSolution}). Further, we
also prove as one of our main results the existence of a Malliavin
differentiable unique solution to a SDE of the type (\ref{SDE0}) in the case
of vector fields $b$ given by a sum of a merely bounded Borel-measurable
function and a Lipschitz continuous function $b:\left[ 0,T\right] \times 
\mathbb{R}\longrightarrow \mathbb{R}$ (Theorem \ref{Ouknine}). We remark
here that the latter result, actually also provides an alternative
construction method of strong solutions in \cite{NO}, where the authors use
a comparison theorem for SDE`s. We also refer to \cite{BMPD} in the Wiener
process case.

Further, we propose a stock price model of Black-Scholes type with a rough
stochastic volatility, whose dynamics is subject to a SDE of the type (\ref%
{SDE0}) and is correlated with the driving noise of the stock price SDE.
Here, we allow for singular drift coefficients, which are either in $L^{1}(%
\mathbb{R}^{d};L^{\infty }([0,T];\mathbb{R}^{d}))$ or a sum of a bounded and
a Lipschitz continuous function. In fact, the selection of a singular or
discontinuous coefficient $b$ in our model is used to capture "regime
switching effects" with respect to the volatility dynamics , which may be
due to financial crashes, market regulations or even natural disasters.
Moreover, we derive with respect to our model, which generalizes the stock
price model in \cite{ACHP} to the case of correlated volatility and stock
price noise and which allows (compared to \cite{ACHP}) for an
Ornstein-Uhlenbeck type of volatility dynamics with regime switching
feature, a BEL-formula for both cases of vector fields $b$ (Theorem \ref%
{Bismut}, Theorem \ref{Volatility} and Theorem \ref{Volatility2}). We
comment on here that such representations are known in finance as greeks
which are sensitivity parameters for measuring e.g. the changes of fair
values of options with respect to the initial underlying stock price (i.e.
delta) or with respect to the initial volatility (i.e. vega). See \cite{F1}, 
\cite{F2} or \cite{DOP} for more information.

Finally, we prove a stability estimate for solutions to (\ref{SDE0}) with
unbounded coefficients (see Proposition \ref{Stability}), which in fact
plays a crucial role in \cite{S1} and \cite{S2} for showing \emph{%
path-by-path }uniqueness of solutions to (\ref{SDE0}) in the case of a
Wiener process and bounded vector fields and which may be also applied to
the situation in our paper. Path-by-path uniqueness, which is a much
stronger concept than pathwise uniqueness and which goes back to \cite{Davie}%
, actually entails the existence of a measurable set $\Omega ^{\ast }$ with
probability mass $1$ such that for all $\omega \in \Omega ^{\ast }$ there
exists a unique deterministic solution $X_{\cdot }^{x}(\omega )$ to (\ref%
{SDE0}) in the space of continuous functions (uniformly in the initial
condition).

\bigskip

Our paper is organized as follows: In Section 2 we prove a BEL-formula
(Theorem \ref{Bismut}) for vector fields $b\in L^{1}(\mathbb{R}%
^{d};L^{\infty }([0,T];\mathbb{R}^{d}))$ based on an existence and a
uniqueness result for strong solutions (Theorem \ref{StrongSolution}). In
Section 3 we introduce our stock price model, with respect to which we
establish BEL-formulas (Theorem \ref{Volatility}, Theorem \ref{Volatility2}%
).\ We also prove the Malliavin differentiability solution of solutions in
the case of vector fields which are a sum of a bounded and Lipschitz
continuous function (Theorem \ref{Ouknine}) and the above mentioned
stability result in this case.

\bigskip

\section{A Bismut-Elworthy-Li formula for integrable vector fields}

\bigskip In this Section we first generalize a result in \cite{BNP} on the
existence of a unique global strong solution of (\ref{SDE0}) to the case of
integrable vector fields $b$. Further, we show that such a solution is
Malliavin differentiable and Sobolev differentiable with respect to the
initial condition, if the Hurst parameter $H$ is small enough. Finally,
based on the latter result, we establish a Bismut-Elworthy-Li formula
(BEL-formula) for solutions to (\ref{SDE0}).

\bigskip

Let $B_{t}^{H},t\geq 0$ be a $d-$dimensional fractional Brownian motion with
Hurst parameter $H\in (0,1)$, that is a $d-$dimensional stochastic process
with components given by independent one-dimensional fractional Brownian
motions with Hurst parameter $H\in (0,1)$ on some complete probability space 
$(\Omega ,\mathcal{F},\mu )$, which are centered Gaussian processes with a
covariance structure $R_{H}(t,s)$ of the form%
\begin{equation*}
R_{H}(t,s)=E[B_{t}^{H}B_{s}^{H}]=\frac{1}{2}(s^{2H}+t^{2H}-\left\vert
t-s\right\vert ^{2H})
\end{equation*}%
for all $t,s\geq 0$. If $H=\frac{1}{2}$ the fractional Brownian motion is a
Wiener process. See the Appendix for more details. We also recall $B_{\cdot
}^{H}$ the representation 
\begin{equation}
B_{t}^{H}=\int_{0}^{t}K_{H}(t,s)I_{d\times d}dB_{s}  \label{RepFractional}
\end{equation}%
for a $d-$dimensional Brownian motion $B_{\cdot }$, where $I_{d\times d}\in 
\mathbb{R}^{d\times d}$ is the unit matrix and $K_{H}$ the kernel as given
in (\ref{KH}) in the Appendix.

Consider the SDE

\begin{equation}
dX_{t}^{x}=b(t,X_{t}^{x})dt+dB_{t}^{H},X_{0}^{x}=x,0\leq t\leq T.
\label{SDE}
\end{equation}

for $H<\frac{1}{2}$.

In what follows, we will also make use of the following function spaces:%
\begin{eqnarray*}
L_{\infty }^{1} &:&=L^{1}(\mathbb{R}^{d};L^{\infty }([0,T];\mathbb{R}^{d})),
\\
L_{\infty }^{\infty } &:&=L^{\infty }(\mathbb{R}^{d};L^{\infty }([0,T];%
\mathbb{R}^{d})), \\
L_{\infty ,\infty }^{1,\infty } &:&=L_{\infty }^{1}\cap L_{\infty }^{\infty
}.
\end{eqnarray*}

\bigskip

In order to prove our first result on a regular unique strong solution to (%
\ref{SDE}), we need the following auxiliary result:

\begin{lemma}
\label{Exp}\bigskip Let $H<\frac{1}{2(d+1)}$ and $b\in L^{\infty }(\left[ 0,T%
\right] ;L^{1}(\mathbb{R}^{d};\mathbb{R}^{d}))$ ($\supset L_{\infty }^{1}$).
Then there exists for all $k\geq 0$ a continuous function $L:\left[ 0,\infty
\right) \longrightarrow \left[ 0,\infty \right) $ depending on $H,T,d$ and $%
k $ such that 
\begin{eqnarray*}
&&E\left[ \exp \left( k\int_{0}^{T}\left\Vert \mathcal{K}_{H}^{-1}(%
\int_{0}^{\cdot }b(s,x+B_{s}^{H})ds)(u)\right\Vert ^{2}du\right) \right] \\
&\leq &L(\left\Vert b\right\Vert _{L^{\infty }(\left[ 0,T\right] ;L^{1}(%
\mathbb{R}^{d}))})\text{,}
\end{eqnarray*}%
where the operator $\mathcal{K}_{H}^{-1}$ is defined as in (\ref{KHminus})
in the Appendix.
\end{lemma}

\begin{proof}
Assume without loss of generality that $b\in L^{\infty }(\left[ 0,T\right]
;L^{1}(\mathbb{R}^{d};\mathbb{R}))$. Using the definition of $\mathcal{K}%
_{H}^{-1}$, we find that 
\begin{eqnarray*}
&&\mathcal{K}_{H}^{-1}(\int_{0}^{\cdot }b(s,x+B_{s}^{H})ds)(u) \\
&=&u^{H-\frac{1}{2}}I_{0+}^{\frac{1}{2}-H}s^{\frac{1}{2}-H}b(s,x+B_{s}^{H})
\\
&=&u^{H-\frac{1}{2}}\frac{1}{\Gamma (\frac{1}{2}-H)}\int_{0}^{u}s^{\frac{1}{2%
}-H}(u-s)^{-\frac{1}{2}-H}b(s,x+B_{s}^{H})ds \\
&=&\frac{1}{\Gamma (\frac{1}{2}-H)}u^{-\frac{1}{2}-H}\int_{0}^{u}(\frac{s}{u}%
)^{\frac{1}{2}-H}(1-\frac{s}{u})^{-\frac{1}{2}-H}b(s,x+B_{s}^{H})ds \\
&&\overset{law}{=}\frac{1}{\Gamma (\frac{1}{2}-H)}u^{-\frac{1}{2}%
-H}\int_{0}^{u}(\frac{s}{u})^{\frac{1}{2}-H}(1-\frac{s}{u})^{-\frac{1}{2}%
-H}b(s,x+u^{H}B_{\frac{s}{u}}^{H})ds \\
&=&\frac{1}{\Gamma (\frac{1}{2}-H)}u^{\frac{1}{2}-H}\int_{0}^{1}s^{\frac{1}{2%
}-H}(1-s)^{-\frac{1}{2}-H}b(su,x+u^{H}B_{s}^{H})ds \\
&=&\frac{1}{\Gamma (\frac{1}{2}-H)}u^{\frac{1}{2}-H}\int_{0}^{1}\gamma _{-%
\frac{1}{2}-H,\frac{1}{2}-H}(1,s)b(su,x+u^{H}B_{s}^{H})ds,
\end{eqnarray*}%
where%
\begin{equation*}
\gamma _{\alpha ,\beta }(u,s):=s^{\beta }(u-s)^{\alpha },u>s\text{.}
\end{equation*}%
So%
\begin{eqnarray*}
&&(\mathcal{K}_{H}^{-1}(\int_{0}^{\cdot }b(s,x+B_{s}^{H})ds)(u))^{2m} \\
&&\overset{law}{=}\frac{1}{\Gamma (\frac{1}{2}-H)^{2m}}u^{2m(\frac{1}{2}%
-H)}(\int_{0}^{1}\gamma _{-\frac{1}{2}-H,\frac{1}{2}%
-H}(1,s)b(su,x+u^{H}B_{s}^{H})ds)^{2m} \\
&=&\frac{1}{\Gamma (\frac{1}{2}-H)^{2m}}u^{2m(\frac{1}{2}-H)}(2m)!\times \\
&&\times \int_{\Delta _{0,1}^{2m}}\gamma _{-\frac{1}{2}-H,\frac{1}{2}%
-H}(1,s_{1})b(s_{1}u,x+u^{H}B_{s_{1}}^{H})...\gamma _{-\frac{1}{2}-H,\frac{1%
}{2}-H}(1,s_{2m})b(s_{2m}u,x+u^{H}B_{s_{2m}}^{H})ds_{1}...ds_{2m}.
\end{eqnarray*}%
Hence,%
\begin{eqnarray*}
&&E\left[ (\int_{0}^{T}(\mathcal{K}_{H}^{-1}(\int_{0}^{\cdot
}b(s,x+B_{s}^{H})ds)(u))^{2}du)^{m}\right] \\
&\leq &T^{m-1}\int_{0}^{T}E\left[ \mathcal{K}_{H}^{-1}(\int_{0}^{\cdot
}b(s,x+B_{s}^{H})ds)(u))^{2m}\right] du \\
&=&T^{m-1}\int_{0}^{T}\frac{1}{\Gamma (\frac{1}{2}-H)^{2m}}u^{2m(\frac{1}{2}%
-H)}(2m)!\times \\
&&\times \int_{\Delta _{0,1}^{2m}}\dprod\limits_{j=1}^{2m}\gamma _{-\frac{1}{%
2}-H,\frac{1}{2}-H}(1,s_{j})E\left[ \dprod%
\limits_{j=1}^{2m}b(s_{j}u,x+u^{H}B_{s_{j}}^{H})\right] ds_{1}...ds_{2m}du.
\end{eqnarray*}%
On the other hand,%
\begin{eqnarray*}
&&E\left[ \dprod\limits_{j=1}^{2m}b(s_{j}u,x+u^{H}B_{s_{j}}^{H})\right] \\
&=&\int_{(\mathbb{R}^{d})^{2m}}\dprod%
\limits_{j=1}^{2m}b(s_{j}u,x+u^{H}y_{j})\dprod\limits_{l=1}^{d}\frac{1}{%
(2\pi )^{m}\det (Q(s))^{m}}\exp (-\frac{1}{2}(y^{(l)})^{\ast
}Q^{-1}(s)y^{(l)})dy^{(1)}...dy^{(d)},
\end{eqnarray*}%
where%
\begin{equation*}
y_{j}:=(y_{j}^{(1)},...,y_{j}^{(d)})^{\ast
},j=1,...,2m,y^{(l)}:=(y_{1}^{(l)},...,y_{2m}^{(l)}),l=1,...,d
\end{equation*}%
and%
\begin{equation*}
Q(s):=Cov\left[ B_{s_{1}}^{H,1},...,B_{s_{2m}}^{H,1}\right] ,
\end{equation*}%
where $B_{u}^{H}=(B_{u}^{H,1},...,B_{u}^{H,d})^{\ast }$. It follows from the
strong local non-determinism of the fractional Brownian motion (see e.g.
Lemma 4.1 and 4.2 in \cite{BLPP}) that%
\begin{equation*}
\det (Q(s))\geq K(H)s_{1}^{2H}(s_{2}-s_{1})^{2H}...(s_{2m}-s_{2m-1})^{2H}
\end{equation*}%
for a constant $K(H)$ not depending on $m$. Thus%
\begin{eqnarray*}
&&\left\vert E\left[ \dprod\limits_{j=1}^{2m}b(s_{j}u,x+u^{H}B_{s_{j}}^{H})%
\right] \right\vert \\
&\leq &\int_{(\mathbb{R}^{d})^{2m}}\dprod\limits_{j=1}^{2m}\left\vert
b(s_{j}u,x+u^{H}y_{j})\right\vert \dprod\limits_{l=1}^{d}\frac{1}{(2\pi
)^{m}\det (Q(s))^{\frac{1}{2}}}\exp (-\frac{1}{2}(y^{(l)})^{\ast
}Q^{-1}(s)y^{(l)})dy^{(1)}...dy^{(d)} \\
&\leq &\frac{1}{(2\pi )^{md}\det (Q(s))^{\frac{d}{2}}}\int_{(\mathbb{R}%
^{d})^{2m}}\dprod\limits_{j=1}^{2m}\left\vert
b(s_{j}u,x+u^{H}y_{j})\right\vert dy^{(1)}...dy^{(d)} \\
&=&\frac{1}{(2\pi )^{md}\det (Q(s))^{\frac{d}{2}}}u^{-2mHd}\int_{(\mathbb{R}%
^{d})^{2m}}\dprod\limits_{j=1}^{2m}\left\vert b(s_{j}u,y_{j})\right\vert
dy_{1}...dy_{2m} \\
&\leq &\frac{1}{(2\pi )^{md}\det (Q(s))^{\frac{d}{2}}}u^{-2mHd}\left(
\sup_{0\leq t\leq T}\int_{\mathbb{R}^{d}}\left\vert b(t,y)\right\vert
dy\right) ^{2m}.
\end{eqnarray*}%
So we obtain that%
\begin{eqnarray*}
&&E\left[ (\int_{0}^{T}(\mathcal{K}_{H}^{-1}(\int_{0}^{\cdot
}b(s,x+B_{s}^{H})ds)(u))^{2}du)^{m}\right] \\
&\leq &T^{m-1}\int_{0}^{T}\frac{1}{\Gamma (\frac{1}{2}-H)^{2m}}u^{2m(\frac{1%
}{2}-H)}(2m)!\times \\
&&\times \int_{\Delta _{0,1}^{2m}}\dprod\limits_{j=1}^{2m}\gamma _{-\frac{1}{%
2}-H,\frac{1}{2}-H}(1,s_{j})\frac{1}{(2\pi )^{md}\det (Q(s))^{\frac{d}{2}}}%
u^{-2mHd}\left( \sup_{0\leq t\leq T}\int_{\mathbb{R}^{d}}\left\vert
b(t,y)\right\vert dy\right) ^{2m}ds_{1}...ds_{2m}du \\
&=&T^{m-1}\left( \sup_{0\leq t\leq T}\int_{\mathbb{R}^{d}}\left\vert
b(t,y)\right\vert dy\right) ^{2m}\int_{0}^{T}\frac{1}{\Gamma (\frac{1}{2}%
-H)^{2m}}u^{2m(\frac{1}{2}-H(d+1))}(2m)!\times \\
&&\times \int_{\Delta _{0,1}^{2m}}\dprod\limits_{j=1}^{2m}\gamma _{-\frac{1}{%
2}-H,\frac{1}{2}-H}(1,s_{j})\frac{1}{(2\pi )^{md}\det (Q(s))^{\frac{d}{2}}}%
ds_{1}...ds_{2m}du.
\end{eqnarray*}%
Further, we know from Lemma A.5 in \cite{BLPP} that%
\begin{eqnarray*}
&&(2m)!\int_{\Delta _{0,1}^{2m}}\dprod\limits_{j=1}^{2m}\gamma _{-\frac{1}{2}%
-H,\frac{1}{2}-H}(1,s_{j})\frac{1}{(2\pi )^{md}\det (Q(s))^{\frac{d}{2}}}%
ds_{1}...ds_{2m} \\
&\leq &C_{H,d}^{m}(m!)^{2H(1+d)}.
\end{eqnarray*}%
Therefore, we get that%
\begin{eqnarray*}
&&E\left[ (\int_{0}^{T}(\mathcal{K}_{H}^{-1}(\int_{0}^{\cdot
}b(s,x+B_{s}^{H})ds)(u))^{2}du)^{m}\right] \\
&\leq &T^{m-1}\frac{1}{\Gamma (\frac{1}{2}-H)^{2m}}\left( \sup_{0\leq t\leq
T}\int_{\mathbb{R}^{d}}\left\vert b(t,y)\right\vert dy\right) ^{2m}TT^{2m(%
\frac{1}{2}-H(d+1))}C_{H,d}^{m}(m!)^{2H(1+d)}.
\end{eqnarray*}%
Hence,%
\begin{eqnarray*}
&&E\left[ \exp (k\int_{0}^{T}(\mathcal{K}_{H}^{-1}(\int_{0}^{\cdot
}b(s,x+B_{s}^{H})ds)(u))^{2}du\right] \\
&\leq &\sum_{m\geq 0}\frac{k^{m}}{m!}C_{H,d}^{m}(m!)^{2H(1+d)}T^{m-1}\frac{1%
}{\Gamma (\frac{1}{2}-H)^{2m}}\left( \sup_{0\leq t\leq T}\int_{\mathbb{R}%
^{d}}\left\vert b(t,y)\right\vert dy\right) ^{2m}TT^{2m(\frac{1}{2}-H(d+1))}
\\
&=&L(\left\Vert b\right\Vert _{L^{\infty }(\left[ 0,T\right] ;L^{1}(\mathbb{R%
}^{d}))})
\end{eqnarray*}%
for a continuous function $L:\left[ 0,\infty \right) \longrightarrow \left[
0,\infty \right) $ depending on $H,T,d$ and $k$.
\end{proof}

\begin{corollary}
\label{WeakSolution}Assume that $H<\frac{1}{2(d+1)}$ and $b\in L^{\infty }(%
\left[ 0,T\right] ;L^{1}(\mathbb{R}^{d};\mathbb{R}^{d}))$. Then, there
exists a weak solution to (\ref{SDE}). Further, suppose that%
\begin{equation*}
(X_{\cdot }^{(i)},B_{\cdot }^{i,H}),(\Omega ^{(i)},\mathcal{F}^{(i)},\mu
^{(i)}),\left\{ \mathcal{F}_{t}^{(i)}\right\} _{0\leq t\leq T},i=1,2
\end{equation*}%
are two weak solutions to (\ref{SDE}) ($B_{\cdot }^{i,H}$ denotes a
fractional Brownian motion with Hurst parameter $H$ with respect to $%
X_{\cdot }^{(i)}$). Require that%
\begin{equation}
\int_{0}^{T}\left\Vert \mathcal{K}_{H}^{-1}(\int_{0}^{\cdot
}b(s,X_{s}^{(i)})ds)(u)\right\Vert ^{2}du<\infty \text{ }\mu ^{(i)}-\text{%
a.e., }i=1,2\text{.}  \label{UniqueWeak}
\end{equation}%
Then both solutions are equal in law.
\end{corollary}

\begin{proof}
The proof of the existence of a weak solution to (\ref{SDE}) is a direct
consequence of Lemma \ref{Exp} in connection with Girsanov `s theorem
(Theorem \ref{girsanov}) and the Novikov condition. The uniqueness of weak
solutions to (\ref{SDE}) follows from a very similar proof of Proposition
5.3.10 in \cite{KaratzasShreve}.
\end{proof}

\bigskip

We are now coming to an existence and uniqueness result of strong solutions
to the SDE (\ref{SDE}) which extends the result in \cite{BNP} for $b\in
L_{\infty ,\infty }^{1,\infty }$ to the case of vector fields in $L_{\infty
}^{1}=L^{1}(\mathbb{R}^{d};L^{\infty }([0,T];\mathbb{R}^{d}))$:

\begin{theorem}
\label{StrongSolution}Suppose that $b\in L_{\infty }^{1}$ and $H<\frac{1}{%
2(d+2)}$. Then there exists a unique (global) strong solution $X_{\cdot
}^{x} $ of the SDE (\ref{SDE}) in the class of stochastic processes
satisfying condition (\ref{UniqueWeak}). Furthermore the solution $X_{\cdot
}^{x}$ is Malliavin differentiable in the direction of the Brownian motion $%
B_{\cdot }$ in (\ref{RepFractional}) at each time point $t\in \lbrack 0,T]$
and for all initial conditions $x\in \mathbb{R}^{d}$.$\ $Moreover, $%
X_{t}^{\cdot }$ is locally Sobolev differentiable $\mu -a.e.,$ that is more
precisely%
\begin{equation*}
X_{t}^{\cdot }\in \dbigcap\limits_{p\geq 2}L^{2}(\Omega ;W^{1,p}(U))
\end{equation*}%
for bounded and open sets $U\subset \mathbb{R}^{d}$.
\end{theorem}

\begin{remark}
\bigskip Compare also the results in \cite{CG} and \cite{NO}, which cannot
be applied to the case, when $b\in L_{\infty }^{1}$ for $d>1$.
\end{remark}

\begin{remark}
We also mention the interesting result in \cite{Le}, which contains the
existence of a unique strong solution in Theorem \ref{StrongSolution} as a
special case. However, the method for proving this result, which is based on
a stochastic sewing lemma and the Yamada-Watanabe approach, doesn%
%TCIMACRO{\U{b4}}%
%BeginExpansion
\'{}%
%EndExpansion
t yield- as in our case- regularity of solutions in the sense of Malliavin
and Sobolev differentiability.
\end{remark}

\begin{remark}
\label{RemarkStrongSolution}Using the same line of reasoning as in the proof
of Theorem \ref{StrongSolution} (see below), we remark that the above result
also holds, if the driving noise $B_{\cdot }^{H}$ is replaced by $\rho
_{1}B_{\cdot }^{H}+\rho _{2}W_{\cdot }$, where $W_{\cdot }$ is a Wiener
process independent of $B_{\cdot }^{H}$ and $\rho _{1},\rho _{2}\in \mathbb{%
R\smallsetminus }\left\{ 0\right\} $.
\end{remark}

\begin{proof}[Proof of Theorem \protect\ref{StrongSolution}]
The proof of this result is very similar to that of \cite{BNP} (see also 
\cite{BLPP}). Therefore, we give here a sketch of the proof, where we
indicate which modifications in the proof are needed.

The idea for the proof of the existence of a unique Malliavin differentiable
solution $X_{\cdot }^{x}$ of the SDE (\ref{SDE}) relies on a compactness
criterion for square integrable functionals of Wiener processes in \cite{DMN}%
. The proof consists of the following steps:

1.step: We consider a sequence of compactly supported smooth functions $%
b_{n}:\left[ 0,T\right] \times \mathbb{R}^{d}\longrightarrow \mathbb{R}%
^{d},n\geq 1$ such that 
\begin{equation*}
b_{n}\underset{n\longrightarrow \infty }{\longrightarrow }b\text{ in }%
L_{\infty }^{1}\text{.}
\end{equation*}%
The objective here is to show that for each $0\leq t\leq T$ the sequence $%
X_{t}^{x,n}$ associated with the SDE%
\begin{equation}
X_{u}^{x,n}=x+\int_{0}^{u}b_{n}(s,X_{s}^{x,n})du+B_{u}^{H},0\leq u\leq T
\label{SDEApprox}
\end{equation}%
for $H<\frac{1}{2(d+2)}$ is relatively compact in $L^{2}(\mu ;\mathbb{R}%
^{d}) $ by employing the compactness criterion for $L^{2}(\mu ;\mathbb{R}%
^{d})$ ($\mu $ Wiener measure) in \cite{DMN}. In fact, the following
estimates in \cite{BNP} give a sufficient criterion for the relative
compactness $X_{t}^{x,n},n\geq 1$:%
\begin{equation}
\left\Vert D_{\cdot }^{B}X_{t}^{x,n}\right\Vert _{L^{2}([0,t]\times \Omega
)}^{2}\leq C_{1}(\left\Vert b_{n}\right\Vert _{L_{\infty ,\infty }^{1,\infty
}})  \label{D1}
\end{equation}%
and%
\begin{equation}
\int_{0}^{t}\int_{0}^{t}\frac{\left\Vert D_{s}^{B}X_{t}^{x,n}-D_{s^{\prime
}}^{B}X_{t}^{x,n}\right\Vert _{L^{2}(\mu )}^{2}}{\left\vert s^{\prime
}-s\right\vert ^{1+2\beta }}dsds^{\prime }\leq C_{2}(\left\Vert
b_{n}\right\Vert _{L_{\infty ,\infty }^{1,\infty }})\text{,}  \label{D2}
\end{equation}%
where $D^{B}$ denotes the Malliavin derivative in the direction of the
Brownian motion $B_{\cdot }$ and where $C_{1},C_{2}:\left[ 0,\infty \right)
\longrightarrow \left[ 0,\infty \right) $ are continuous functions depending
on $H,T$ and $d$. By using Girsanov`s theorem in connection with Lemma \ref%
{Exp} (see for more detailed explanations the next steps) we obtain under
the assumptions of Theorem \ref{StrongSolution} the same estimates as in (%
\ref{D1}), (\ref{D2}) with the only difference that $\left\Vert
b_{n}\right\Vert _{L_{\infty ,\infty }^{1,\infty }}$ in the functions $%
C_{1},C_{2}$ is replaced by $\left\Vert b_{n}\right\Vert _{L_{\infty }^{1}}$.

2.step: By applying the Malliavin derivative $D^{B}$ to both sides of (\ref%
{SDEApprox}) and the chain rule for the Malliavin derivative (see \cite%
{Nualart} or \cite{DOP}), we find that%
\begin{equation}
D_{\theta }^{B}X_{u}^{x,n}=x+\int_{\theta }^{u}b_{n}^{\shortmid
}(s,X_{s}^{x,n})D_{\theta }^{B}X_{s}^{x,n}du+K_{H}(u,\theta )I_{d\times
d},0\leq \theta <u\leq T,n\geq 1,a.e.,
\end{equation}%
where $b_{n}^{\shortmid }$ is the spatial Fr\'{e}chet derivative of $b_{n}$.
Then, Picard iteration gives%
\begin{eqnarray}
D_{\theta }^{B}X_{u}^{x,n} &=&K_{H}(u,\theta )I_{d\times d}+  \notag \\
&&\sum_{m\geq 1}\int_{\theta <s_{1}<...<s_{m}<u}b_{n}^{\shortmid
}(s_{m},X_{s_{m}}^{x,n})...b_{n}^{\shortmid
}(s_{1},X_{s_{1}}^{x,n})K_{H}(s_{1},\theta )ds_{1}...ds_{m},  \label{Picard}
\end{eqnarray}%
where the convergence is in $L^{p}$-sense. Then, in order to "eliminate" the
derivatives $b_{n}^{\shortmid }$ in (\ref{Picard}) we can apply Girsanov`s
change of measure (see Theorem \ref{girsanov} in the Appendix) combined with
Lemma \ref{Exp} and the following "local time variational calculus argument":%
\begin{equation}
\int_{\Delta _{\theta ,t}^{m}}D^{\alpha }f(s,B_{s}^{H})ds=\int_{\left( 
\mathbb{R}^{d}\right) ^{m}}\Lambda _{\alpha }^{f}(\theta ,t,z)dz
\label{ibp0}
\end{equation}%
for a random field $\Lambda _{\alpha }^{f}$ which (in the case of
time-homogeneous vector fields) can be interpreted as a (scaled) local time
on the $m$-dimensional simplex%
\begin{equation}
\Delta _{\theta ,t}^{m}:=\left\{ (s_{m},...,s_{1})\in \left[ 0,T\right]
^{m}:\theta <s_{m}<...<s_{1}<t\right\} .  \label{Simplex}
\end{equation}%
See (\ref{ibp}) in the Appendix for the precise definitions of the notation
involved. Here we apply (\ref{ibp0}) to the case, when%
\begin{equation*}
D^{\alpha }f(s,z)=\dprod\limits_{j=1}^{m}D^{\alpha
_{j}}f_{j}(s_{j},z_{j})=b_{n}^{\shortmid }(s_{m},x+z_{m})...b_{n}^{\shortmid
}(s_{1},x+z_{1})K_{H}(s_{1},\theta )\text{.}
\end{equation*}%
Finally, certain estimates with respect to $\Lambda _{\alpha }^{f}$ (see
Theorem \ref{mainestimate} in the Appendix) yield the bounds (\ref{D1}), (%
\ref{D2}) (with $\left\Vert b_{n}\right\Vert _{L_{\infty ,\infty }^{1,\infty
}}$ in the functions $C_{1},C_{2}$ is replaced by $\left\Vert
b_{n}\right\Vert _{L_{\infty }^{1}}$).

3. step: The bounds of the type (\ref{D1}), (\ref{D2}) in step 2 enable us
to apply the compactness criterion in \cite{DMN} and we obtain that%
\begin{equation*}
X_{t}^{x,n_{l}(t)}\underset{l\longrightarrow \infty }{\longrightarrow }Y_{t}%
\text{ in }L^{2}(\mu )
\end{equation*}%
for a subsequence $n_{l}(t),l\geq 1$. However, by using a very similar proof
of Lemma 5.5 in \cite{BLPP} in connection with Lemma \ref{Exp} (or see \cite%
{BNP}) it turns out that%
\begin{equation*}
Y_{t}=E\left[ X_{t}^{x}\right. \left\vert \mathcal{F}_{t}\right] \text{,}
\end{equation*}%
where $X_{\cdot }^{x}$ is the weak solution to (\ref{SDE}) of Corollary \ref%
{WeakSolution} and where $\mathcal{F}_{t},0\leq t\leq T$ is the (augmented)
filtration generated by $B_{\cdot }^{H}$. Thus%
\begin{equation}
X_{t}^{x,n}\underset{n\longrightarrow \infty }{\longrightarrow }Y_{t}\text{
in }L^{2}(\mu )  \label{ConvergenceX}
\end{equation}%
Using the latter, one shows for all bounded and continuous functions $%
\varphi $ that%
\begin{equation*}
\varphi (E\left[ X_{t}^{x}\right. \left\vert \mathcal{F}_{t}\right] )=E\left[
\varphi (X_{t}^{x})\right. \left\vert \mathcal{F}_{t}\right] \text{ }a.e.,
\end{equation*}%
which implies the $\mathcal{F}_{t}$-adaptedness of $X_{t}^{x}$. Hence, the
weak solution in Corollary \ref{WeakSolution} is a strong solution. Strong
uniqueness of solutions of (\ref{SDE}) is also a consequence of Corollary %
\ref{WeakSolution}. Further, Malliavin differentiability of the solution
directly follows from the compactness criterion in \cite{DMN}.

As for the assertion of the Sobolev regularity of the solution with respect
to the initial condition, one can invoke the "local time variational
argument" combined with Girsanov`s theorem (Theorem \ref{girsanov}) and
Lemma \ref{Exp} in step 2 in the same way and derive the following estimate:
For $p\geq 2$ and $H<\frac{1}{2(d+2)},$ we have%
\begin{equation}
\sup_{x\in \mathbb{R}^{d}}E[\left\Vert \frac{\partial }{\partial x}%
X_{t}^{x,n}\right\Vert ^{p}]\leq C_{p,H,d,T}(\left\Vert b_{n}\right\Vert
_{L_{\infty }^{\infty }},\left\Vert b_{n}\right\Vert _{L_{\infty
}^{1}})<\infty ,n\geq 1  \label{BoundedDerivatives}
\end{equation}%
for some continuous function $C_{p,H,d,T}:[0,\infty )^{2}\longrightarrow
\lbrack 0,\infty )$. However, because of Lemma \ref{Exp} we obtain under
assumptions of Theorem \ref{StrongSolution} the bound%
\begin{equation}
\sup_{x\in \mathbb{R}^{d}}E[\left\Vert \frac{\partial }{\partial x}%
X_{t}^{x,n}\right\Vert ^{p}]\leq L_{p,H,d,T}(\left\Vert b_{n}\right\Vert
_{L_{\infty }^{1}})<\infty ,n\geq 1  \label{DerivativeEstimate}
\end{equation}%
for some continuous function $L_{p,H,d,T}:[0,\infty )\longrightarrow \lbrack
0,\infty )$, which yields the spatial regularity of the solution.
\end{proof}

\bigskip

The next result is a generalization of that in \cite{ACHP} to the case of
unbounded vector fields in $L_{\infty }^{1}=L^{1}(\mathbb{R}^{d};L^{\infty
}([0,T];\mathbb{R}^{d}))$:

\begin{theorem}[Bismut-Elworthy-Li formula for $b\in L_{\infty }^{1}$]
\label{Bismut}Assume that $b\in L_{\infty }^{1}$ and $H<\frac{1}{2(d+2)}$.
Let $X_{\cdot }^{x}$ be the unique strong solution to the SDE%
\begin{equation*}
dX_{t}^{x}=b(t,X_{t}^{x})dt+dB_{t}^{H},X_{0}^{x}=x,0\leq t\leq T.
\end{equation*}%
Suppose that $U$ is a bounded and open subset of $\mathbb{R}^{d}$ and $\Phi :%
\mathbb{R}^{d}\longrightarrow \mathbb{R}$ a Borel measurable function with%
\begin{equation*}
\Phi (X_{T}^{\cdot })\in L^{2}(\Omega \times U,\mu \times dx).
\end{equation*}%
Further, let $a:[0,T]\longrightarrow \mathbb{R}$ be a bounded Borel
measurable function such that\ 
\begin{equation*}
\int_{0}^{T}a(s)ds=1.
\end{equation*}%
Then, we have the following representation%
\begin{eqnarray}
&&\frac{\partial }{\partial x}E[\Phi (X_{T}^{x})]  \notag \\
&=&C_{H}E[\Phi (X_{T}^{x})\int_{0}^{T}u^{-H-\frac{1}{2}%
}\int_{u}^{T}a(s-u)(s-u)^{\frac{1}{2}-H}s^{H-\frac{1}{2}}\left( \frac{%
\partial }{\partial x}X_{s-u}^{x}\right) ^{\ast }dB_{s}du]^{\ast }
\label{BEL}
\end{eqnarray}%
for all $x\in U$ a.e., $0<t\leq T$, where $\ast $ is the transpose of a
matrix and where $C_{H}=1/(c_{H}\Gamma (\frac{1}{2}+H)\Gamma (\frac{1}{2}%
-H)) $ for%
\begin{equation*}
c_{H}=(\frac{2H}{(1-2H)B(1-2H,H+1/2)})^{1/2}.
\end{equation*}%
Here $\Gamma $ and $B$ are the Gamma and Beta function, respectively.
\end{theorem}

\begin{proof}
The proof is very similar to that in \cite{ACHP}. However, since we will
need parts of the proof in\ Section 3, it can be found in the Appendix.
\end{proof}

\begin{remark}
We mention that $\frac{\partial }{\partial x}X_{t}^{x},0\leq t\leq T$ in the
BEL-formula is a process $Y:[0,T]\times \Omega \times U\longrightarrow 
\mathbb{R}^{d\times d}$ in $L^{2}([0,T]\times \Omega \times U,\mathcal{P}%
\otimes \mathcal{B}(U);\mathbb{R}^{d\times d})$ such that $Y_{t}^{\cdot
}(\omega )$ coincides with the Sobolev derivative of $X_{t}^{\cdot }(\omega
) $ $(t,\omega )-$a.e. Here, $\mathcal{P}$ denotes the predictable $\sigma -$%
algebra with respect to the $\mu -$augmented filtration $\{\mathcal{F}%
_{t}\}_{0\leq t\leq T}$ generated by $B_{\cdot }^{H}$.
\end{remark}

\begin{remark}
\bigskip From a financial mathematics point of view the expression on the
right hand side of (\ref{BEL}) has the interpretation of the greek delta,
that is a sensitivity measure, which measures changes of the fair value of a
financial claim with payoff function $\Phi $ and underlying $d$ stock price
processes $X_{\cdot }^{x}$ with respect to the initial prices $x\in \mathbb{R%
}^{d}$ of the stocks. However, since $\frac{\partial }{\partial x}X_{t}^{x}$
is a Sobolev derivative, this sensitivity measure is only defined for $x\in U
$ a.e., which makes it rather unusable in financial applications. In order
to overcome this problem, one may choose as in \cite{ACHP} a continuous
version of $\frac{\partial }{\partial x}X_{t}^{x}$ by using the following
estimate, which can derived (by means of Girsanov`s theorem and Lemma \ref%
{Exp}) in the same way as in Proposition 10, \cite{ACHP} in the case of $%
b\in L_{\infty ,\infty }^{1,\infty }$: Let $b\in C_{c}^{\infty }((0,T)\times 
\mathbb{R}^{d})$ and $p>2$. Then, if $H<\frac{1}{2(d+3)},$ we have 
\begin{equation}
\sup_{x\in \mathbb{R}^{d}}E[\left\Vert \frac{\partial ^{2}}{\partial x^{2}}%
X_{t}^{s,x}\right\Vert ^{p}]\leq C_{p,H,d,T}(\left\Vert b\right\Vert
_{L_{\infty }^{1}})<\infty   \label{SecondD}
\end{equation}%
for some continuous function $C_{p,H,d,T}:[0,\infty )\longrightarrow \lbrack
0,\infty )$.
\end{remark}

\section{\protect\bigskip A regularity result for SDE `s with
non-integrable, unbounded vector fields and a stock price model with regime
switching correlated rough volatility}

As one of the main results (Theorem \ref{Ouknine}) in this Section we prove
the Malliavin differentiability to (\ref{SDE}) for $d=1$, when the vector
field $b$ is given by a sum of a bounded and Lipschitz continuous function.
This result, whose proof is based on the compactness in \cite{DMN} and a
transfer principle beween a Wiener process and a fractional Brownian motion,
provides a generalization of Theorem 3.1 in \cite{BMPD} in the case of
fractional noise for $H<\frac{1}{2}$. Further, we introduce an extension of
a stock price model with regime switching rough volatility in \cite{ACHP},
which allows for correlation between the driving noise of the stock price
SDE and that of the stochastic volatility. Then, we derive (as in \cite{ACHP}%
) a BEL-formula with respect to the this model. Here, we first study the
case, when the drift $b$ of the SDE for the volatility process belongs to $%
L_{\infty }^{1}$. Then, we examine the case, when $b$ can be decomposed as a
sum of a bounded and Lipschitz continuous function. Finally, we conclude the
paper with the proof of a stability result for solutions to (\ref{SDE}) in
the setting of Theorem \ref{Ouknine}.

\bigskip

Let us now consider the following model for stock prices $%
S_{t}^{x_{1},x_{2}},0\leq t\leq T$ with stochastic volatility $\sigma
_{t}^{x_{2}},0\leq t\leq T$ given by the solution to the SDE%
\begin{eqnarray}
S_{t}^{x_{1},x_{2}} &=&x_{1}+\int_{0}^{t}\mu
S_{u}^{x_{1},x_{2}}du+\int_{0}^{t}g(\sigma
_{u}^{x_{2}})S_{u}^{x_{1},x_{2}}dW_{u}  \notag \\
\sigma _{t}^{x_{2}} &=&x_{2}+\int_{0}^{t}b(u,\sigma _{u}^{x_{2}})du+\sqrt{%
1-\rho ^{2}}B_{t}^{H}+\rho W_{t},x_{1},x_{2}\in \mathbb{R},0\leq t\leq T,
\label{Stock}
\end{eqnarray}%
where $W_{\cdot \text{ }}$ is a Wiener process, which is independent of a
fractional Brownian motion $B_{\cdot }^{H}$ with Hurst parameter $H<\frac{1}{%
2(d+2)}=\frac{1}{6}$ for $d=1$. Here $\mu \in \mathbb{R}$, $b\in L_{\infty
}^{1}$ and $g:\mathbb{R}\longrightarrow (\alpha ,\infty )$ is a function
from $C_{b}^{2}(\mathbb{R})$ for some $\alpha >0$. Further, $\rho \in (-1,1)$
describes the correlation between $B_{\cdot }^{H}$ and $W_{\cdot \text{ }}$.

We mention that the stock price model (\ref{Stock}) covers that in \cite%
{ACHP} as a special case, when $\rho =0$.

In order to establish a BEL-formula for the stock price model (\ref{Stock})
in the case of vector fields $b\in L_{\infty }^{1}$, we shall follow here
cosely the arguments and exposition in \cite{ACHP}:

In the sequel, let $\Omega =\Omega _{1}\times \Omega _{2}$ for sample spaces 
$\Omega _{1}$, $\Omega _{2}$, which accommodate $W_{\cdot \text{ }}$ and $%
B_{\cdot }^{H}$.

For the time being, require that $b\in C_{c}^{\infty }((0,T)\times \mathbb{R}%
^{d}).$ Then, one can show (see e.g. \cite{Nualart}) that $%
X_{t}^{x}:=(S_{t}^{x_{1},x_{2}},\sigma _{t}^{x_{2}})^{\ast },x=(x_{1},x_{2})$
is Malliavin differentiable with respect to $Z=(Z^{(1)},Z^{(2)})^{\ast
}=(W,B^{H})^{\ast }$ with Malliavin derivative $D=(D^{W},D^{H})^{\ast }$ and
we obtain that%
\begin{eqnarray*}
&&D_{s}X_{t}^{x} \\
&=&\int_{s}^{t}\left( 
\begin{array}{ll}
\mu & 0 \\ 
0 & b^{\shortmid }(u,\sigma _{u}^{x_{2}})%
\end{array}%
\right) D_{s}X_{u}^{x}du \\
&&+\left( \sum_{j=1}^{2}\int_{s}^{t}\sum_{l=1}^{2}\frac{\partial }{\partial
x_{l}}a_{ij}(S_{u}^{x_{1},x_{2}},\sigma
_{u}^{x_{2}})(D_{s}X_{u}^{x})_{rl}dZ_{u}^{(j)}\right) _{1\leq i,r\leq 2} \\
&&+\chi _{_{\lbrack 0,t]}(s)}\left( a_{ij}(S_{s}^{x_{1},x_{2}},\sigma
_{s}^{x_{2}})\right) _{1\leq i,j\leq 2} \\
&=&\int_{s}^{t}\left( 
\begin{array}{ll}
\mu & 0 \\ 
0 & b^{\shortmid }(u,\sigma _{u}^{x_{2}})%
\end{array}%
\right) D_{s}X_{u}^{x}du \\
&&+\left( \int_{s}^{t}\sum_{l=1}^{2}\frac{\partial }{\partial x_{l}}%
a_{i1}(S_{u}^{x_{1},x_{2}},\sigma
_{u}^{x_{2}})(D_{s}X_{u}^{x})_{rl}dW_{u}\right) _{1\leq i,r\leq 2} \\
&&+\chi _{_{\lbrack 0,t]}(s)}\left( a_{ij}(S_{s}^{x_{1},x_{2}},\sigma
_{s}^{x_{2}})\right) _{1\leq i,j\leq 2}
\end{eqnarray*}%
where 
\begin{equation*}
\left( a_{ij}(x_{1},x_{2})\right) _{1\leq i,j\leq 2}=\left( 
\begin{array}{ll}
g(x_{2})x_{1} & 0 \\ 
\rho & \sqrt{1-\rho ^{2}}%
\end{array}%
\right) .
\end{equation*}%
Using very similar arguments as e.g. in \cite{Kunita}, we find that $%
X_{t}^{x,y}$ is twice continuously differentiable with respect to $(x,y)$.
Then, by following a similar line of reasoning as in the proof of Theorem %
\ref{Bismut} in combination with a substitution formula for Wiener integrals 
\cite[Theorem 3.2.9]{Nualart}, we get that

\begin{equation*}
D_{s}X_{t}^{x}=\frac{\partial }{\partial x}X_{t}^{s,X_{s}^{x}}\chi
_{_{\lbrack 0,t]}(s)}\left( a_{ij}(S_{s}^{x_{1},x_{2}},\sigma
_{s}^{x_{2}})\right) _{1\leq i,j\leq 2}.
\end{equation*}%
In a similar manner, we observe that 
\begin{equation*}
\frac{\partial }{\partial x}E[\Phi (X_{T}^{x,n})]=E[\Phi ^{\shortmid
}(X_{T}^{x})\frac{\partial }{\partial x}X_{T}^{s,X_{s}^{x}}\frac{\partial }{%
\partial x}X_{s}^{x}]
\end{equation*}%
for payoff functions $\Phi \in C_{c}^{\infty }(\mathbb{R}^{2})$. Hence,%
\begin{equation*}
\frac{\partial }{\partial x}E[\Phi (X_{T}^{x})]=E[\Phi ^{\shortmid
}(X_{T}^{x})D_{s}X_{T}^{x}\left( a_{ij}(S_{s}^{x_{1},x_{2}},\sigma
_{s}^{x_{2}})\right) _{1\leq i,j\leq 2}^{-1}\frac{\partial }{\partial x}%
X_{s}^{x}].
\end{equation*}%
So, for $a\in L^{\infty }(\left[ 0,T\right] )$ with $\int_{0}^{T}a(s)ds=1$
we can employ the chain rule with respect to $D_{\cdot \text{ }}$ and get
that%
\begin{eqnarray*}
&&\frac{\partial }{\partial x}E[\Phi (X_{T}^{x})] \\
&=&E[\int_{0}^{T}\{a(s)\Phi ^{\shortmid }(X_{T}^{x})D_{s}X_{T}^{x}\left(
a_{ij}(S_{s}^{x_{1},x_{2}},\sigma _{s}^{x_{2}})\right) _{1\leq i,j\leq
2}^{-1}\frac{\partial }{\partial x}X_{s}^{x}\}ds] \\
&=&E[\int_{0}^{T}\{a(s)D_{s}\Phi (X_{T}^{x})\left(
a_{ij}(S_{s}^{x_{1},x_{2}},\sigma _{s}^{x_{2}})\right) _{1\leq i,j\leq
2}^{-1}\frac{\partial }{\partial x}X_{s}^{x}\}ds]
\end{eqnarray*}%
We also see that%
\begin{eqnarray*}
&&\left( a_{ij}(S_{s}^{x_{1},x_{2}},\sigma _{s}^{x_{2}})\right) _{1\leq
i,j\leq 2}^{-1}\frac{\partial }{\partial x}X_{s}^{x} \\
&=&(S_{s}^{x_{1},x_{2}}g(\sigma _{s}^{x_{2}})\sqrt{1-\rho ^{2}})^{-1} \\
&&\cdot \left( 
\begin{array}{ll}
\sqrt{1-\rho ^{2}}\frac{\partial }{\partial x_{1}}S_{s}^{x_{1},x_{2}} & 
\sqrt{1-\rho ^{2}}\frac{\partial }{\partial x_{2}}S_{s}^{x_{1},x_{2}} \\ 
-\rho \frac{\partial }{\partial x_{1}}S_{s}^{x_{1},x_{2}} & -\rho \frac{%
\partial }{\partial x_{2}}\sigma _{s}^{x_{2}}+\frac{\partial }{\partial x_{2}%
}\sigma _{s}^{x_{2}}S_{s}^{x_{1},x_{2}}g(\sigma _{s}^{x_{2}})%
\end{array}%
\right) .
\end{eqnarray*}%
So it follows that%
\begin{eqnarray*}
&&D_{s}\Phi (X_{T}^{x})\left( a_{ij}(S_{s}^{x_{1},x_{2}},\sigma
_{s}^{x_{2}})\right) _{1\leq i,j\leq 2}^{-1}\frac{\partial }{\partial x}%
X_{s}^{x} \\
&=&(D_{s}^{W}\Phi (X_{T}^{x})(S_{s}^{x_{1},x_{2}}g(\sigma _{s}^{x_{2}}))^{-1}%
\frac{\partial }{\partial x_{1}}S_{s}^{x_{1},x_{2}} \\
&&-D_{s}^{H}\Phi (X_{T}^{x})\frac{\rho }{\sqrt{1-\rho ^{2}}}%
(S_{s}^{x_{1},x_{2}}g(\sigma _{s}^{x_{2}}))^{-1}\frac{\partial }{\partial
x_{1}}S_{s}^{x_{1},x_{2}}, \\
&&D_{s}^{W}\Phi (X_{T}^{x})(S_{s}^{x_{1},x_{2}}g(\sigma _{s}^{x_{2}}))^{-1}%
\frac{\partial }{\partial x_{2}}S_{s}^{x_{1},x_{2}} \\
&&+D_{s}^{H}\Phi (X_{T}^{x})(\frac{-\rho }{\sqrt{1-\rho ^{2}}}\frac{\partial 
}{\partial x_{2}}S_{s}^{x_{1},x_{2}}(S_{s}^{x_{1},x_{2}}g(\sigma
_{s}^{x_{2}}))^{-1} \\
&&+\frac{1}{\sqrt{1-\rho ^{2}}}\frac{\partial }{\partial x_{2}}\sigma
_{s}^{x_{2}}))^{\ast }.
\end{eqnarray*}%
Therefore, we obtain that%
\begin{eqnarray*}
&&\frac{\partial }{\partial x}E[\Phi (X_{T}^{x})] \\
&=&(E[\int_{0}^{T}a(s)(D_{s}^{W}\Phi (X_{T}^{x})(S_{s}^{x_{1},x_{2}}g(\sigma
_{s}^{x_{2}}))^{-1}\frac{\partial }{\partial x_{1}}S_{s}^{x_{1},x_{2}}ds] \\
&&-E[\int_{0}^{T}a(s)D_{s}^{H}\Phi (X_{T}^{x})\frac{\rho }{\sqrt{1-\rho ^{2}}%
}(S_{s}^{x_{1},x_{2}}g(\sigma _{s}^{x_{2}}))^{-1}\frac{\partial }{\partial
x_{1}}S_{s}^{x_{1},x_{2}}ds], \\
&&E[\int_{0}^{T}a(s)(D_{s}^{W}\Phi (X_{T}^{x})(S_{s}^{x_{1},x_{2}}g(\sigma
_{s}^{x_{2}}))^{-1}\frac{\partial }{\partial x_{2}}S_{s}^{x_{1},x_{2}}ds] \\
&&+E[\int_{0}^{T}a(s)D_{s}^{H}\Phi (X_{T}^{x})(\frac{-\rho }{\sqrt{1-\rho
^{2}}}\frac{\partial }{\partial x_{2}}%
S_{s}^{x_{1},x_{2}}(S_{s}^{x_{1},x_{2}}g(\sigma _{s}^{x_{2}}))^{-1} \\
&&+\frac{1}{\sqrt{1-\rho ^{2}}}\frac{\partial }{\partial x_{2}}\sigma
_{s}^{x_{2}})ds])^{\ast }.
\end{eqnarray*}%
The by exploiting the independence of $W_{\cdot }$ and $B_{\cdot }^{H}$, we
can use a similar reasoning as in the proof of Theorem \ref{Bismut} and find
that%
\begin{eqnarray*}
&&E[\int_{0}^{T}a(s)D_{s}^{H}\Phi (X_{T}^{x})(\frac{-\rho }{\sqrt{1-\rho ^{2}%
}}\frac{\partial }{\partial x_{2}}S_{s}^{x_{1},x_{2}}(S_{s}^{x_{1},x_{2}}g(%
\sigma _{s}^{x_{2}}))^{-1} \\
&&+\frac{1}{\sqrt{1-\rho ^{2}}}\frac{\partial }{\partial x_{2}}\sigma
_{s}^{x_{2}})ds] \\
&=&C_{H}E[\Phi (X_{T}^{x})\int_{0}^{T}u^{-H-\frac{1}{2}%
}\int_{u}^{T}a(s-u)(s-u)^{\frac{1}{2}-H}s^{H-\frac{1}{2}} \\
&&\cdot \{\frac{-\rho }{\sqrt{1-\rho ^{2}}}\frac{\partial }{\partial x_{2}}%
S_{s-u}^{x_{1},x_{2}}(S_{s-u}^{x_{1},x_{2}}g(\sigma _{s-u}^{x_{2}}))^{-1} \\
&&+\frac{1}{\sqrt{1-\rho ^{2}}}\frac{\partial }{\partial x_{2}}\sigma
_{s-u}^{x_{2}}\}dB_{s}du],
\end{eqnarray*}%
where $B_{\cdot }$ is a one-dimensional Brownian motion in the stochastic
integral representation (\ref{RepFractional}).

In the last step, we can use the duality formula with respect to $W_{\cdot }$
and similar arguments as in the proof of Theorem \ref{Bismut} (see also
Remark \ref{RemarkStrongSolution}) with respect to regular functions $g,$ $b,
$ $\Phi $ and derive the following BEL-formula for the stock price model (%
\ref{Stock}):

\begin{theorem}
\label{Volatility}Assume that $U\subset \mathbb{R}^{2}$ is a bounded, open
set and $b\in L_{\infty }^{1}$ in the stock price model (\ref{Stock}). In
addition, require that $g:\mathbb{R}\longrightarrow (\alpha ,\infty )$ is in 
$C_{b}^{2}(\mathbb{R})$ for some $\alpha >0$ and that $\Phi :\mathbb{R}%
^{2}\longrightarrow \mathbb{R}$ is a Borel-measurable function with%
\begin{equation*}
\Phi (S_{T}^{\cdot ,\cdot },\sigma _{T}^{\cdot })\in L^{2}(\Omega \times
U,\mu \times dx)\text{.}
\end{equation*}%
Let $a\in L^{\infty }(\left[ 0,T\right] )$ with $\int_{0}^{T}a(s)ds=1$. Then%
\begin{eqnarray}
&&\frac{\partial }{\partial x}E[\Phi (S_{T}^{x_{1},x_{2}},\sigma
_{T}^{x_{2}})]  \notag \\
&=&(E[\Phi (X_{T}^{x})\int_{0}^{T}a(s)(S_{s}^{x_{1},x_{2}}g(\sigma
_{s}^{x_{2}}))^{-1}\frac{\partial }{\partial x_{1}}S_{s}^{x_{1},x_{2}}dW_{s}]
\notag \\
&&-C_{H}E[\Phi (X_{T}^{x})\int_{0}^{T}u^{-H-\frac{1}{2}%
}\int_{u}^{T}a(s-u)(s-u)^{\frac{1}{2}-H}s^{H-\frac{1}{2}}  \notag \\
&&\cdot \{\frac{\rho }{\sqrt{1-\rho ^{2}}}(S_{s-u}^{x_{1},x_{2}}g(\sigma
_{s-u}^{x_{2}}))^{-1}\frac{\partial }{\partial x_{1}}S_{s-u}^{x_{1},x_{2}}%
\}dB_{s}du],  \notag \\
&&E[\Phi (X_{T}^{x})\int_{0}^{T}a(s)(S_{s}^{x_{1},x_{2}}g(\sigma
_{s}^{x_{2}}))^{-1}\frac{\partial }{\partial x_{2}}S_{s}^{x_{1},x_{2}}dW_{s}]
\notag \\
&&+C_{H}E[\Phi (X_{T}^{x})\int_{0}^{T}u^{-H-\frac{1}{2}%
}\int_{u}^{T}a(s-u)(s-u)^{\frac{1}{2}-H}s^{H-\frac{1}{2}}  \notag \\
&&\cdot \{\frac{-\rho }{\sqrt{1-\rho ^{2}}}\frac{\partial }{\partial x_{2}}%
S_{s-u}^{x_{1},x_{2}}(S_{s-u}^{x_{1},x_{2}}g(\sigma _{s-u}^{x_{2}}))^{-1} 
\notag \\
&&+\frac{1}{\sqrt{1-\rho ^{2}}}\frac{\partial }{\partial x_{2}}\sigma
_{s-u}^{x_{2}}\}dB_{s}du])^{\ast }  \label{SBEL}
\end{eqnarray}%
for almost all $x=(x_{1},x_{2})\in U$, where $C_{H}$ is a constant defined
as in Theorem \ref{Bismut}.
\end{theorem}

\begin{remark}
In fact, one can prove as in \cite{ACHP} by using the estimate \ref{SecondD}
that the right hand side of (\ref{SBEL}) has a continuous modification, if $%
H<\frac{1}{2(d+3)}=\frac{1}{8}$ (for $d=1$).
\end{remark}

\bigskip

We mention that our stock price model (\ref{Stock}) coincides with the model
of Gatheral, Jaisson, Rosenbaum \cite{G} in the case of independent $\sigma
_{t}^{x_{2}}$, $W_{t},0\leq t\leq T$, when (formally)%
\begin{equation*}
g(x)=\exp (x/\nu ),b(t,x)=-a\nu (x-b)\text{ and }\rho =0
\end{equation*}%
or more explicitly, when 
\begin{eqnarray}
S_{t}^{x_{1},x_{2}} &=&x_{1}+\int_{0}^{t}\mu
S_{u}^{x_{1},x_{2}}du+\int_{0}^{t}\exp (\sigma _{u}^{x_{2}}/\nu
)S_{u}^{x_{1},x_{2}}dW_{u}  \notag \\
\sigma _{t}^{x_{2}} &=&x_{2}/\nu -\int_{0}^{t}a\nu (\sigma
_{u}^{x_{2}}-b)du+B_{t}^{H},x_{1},x_{2}\in \mathbb{R},0\leq t\leq T,
\label{fOU}
\end{eqnarray}%
for $a,\nu >0,b\in \mathbb{R}$.

The process $\sigma _{t}^{x_{2}},0\leq t\leq T$ in (\ref{fOU}) is a
stationary and mean reverting process, which can be regarded as a
generalization of the Vasicek model for short rates in the case of
stochastic log-volatility with (log-volatility) mean reversion $a\nu $ and
long-run average level $b$.

It turns out that $\sigma _{t}^{x_{2}},0\leq t\leq T$ has the explicit
representation%
\begin{equation*}
\sigma _{t}^{x_{2}}=(x_{2}/\nu )+b(1-e^{-a\nu t})+\int_{0}^{t}e^{-a\nu
(t-s)}dB_{s}^{H},
\end{equation*}%
where the last term on the right hand side is defined as a Young integral
with respect to the integrator $B_{s}^{H}$ for $H<\frac{1}{2}$.\ See \cite%
{Young}.

However, due to economical crises, financial disasters or market regulations
one would expect to observe a change with respect to the behaviour of the
dynamics $\sigma _{t}^{x_{2}}$, that is e.g. a "regime change" from a
log-volatility long-run average level $b_{1}$ to $b_{2}$ or from a mean
reversion $a_{1}$ to $a_{2}$, provided $\sigma _{t}^{x_{2}}$ exceeds a
certain threshold $R$. In order to capture such "regime switching" effects
and the roughness of paths of $\sigma _{t}^{x_{2}}$, whose empirical
evidence was found in \cite{G} and which is modelled by means of $B_{t}^{H}$
for small Hurst parameters $H<\frac{1}{2}$, it is natural to take the
stochastic volatility model (\ref{fOU}) as a starting point and to modify it
as follows:%
\begin{equation}
\sigma _{t}^{x_{2}}=x_{2}/\nu -\int_{0}^{t}b(u,\sigma
_{u}^{x_{2}})du+B_{t}^{H},0\leq t\leq T,  \label{RSfOU}
\end{equation}%
where $b$ is a \emph{discontinuous} vector field of linear growth given by%
\begin{equation}
b(t,y)=a_{1}(y-b_{1})1_{(-\infty ,R)}(y)+a_{2}(y-b_{2})1_{[R,\infty )}(y)
\label{RS}
\end{equation}%
for some $a_{1},a_{2}>0,R,b_{1},b_{2}\in \mathbb{R}.$

If $g(x)=\exp (x)$ and if $\sigma _{t}^{x_{2}}$ follows the dynamics in (\ref%
{RSfOU}), one observes that the stock price process $S_{t}^{x_{1},x_{2}}$ in
(\ref{Stock}) is not square integrable in general and hence not Malliavin
differentiable. Therefore we cannot directly use Malliavin techniques here
to derive a BEL-formula as in Theorem \ref{Volatility} for this situation.
In order to overcome this deficiency, one may in view of applications
instead replace the exponential function by a function $g:\mathbb{%
R\longrightarrow }\mathbb{R}$ such that%
\begin{equation}
g(x)=\exp (f(x)),  \label{Truncation}
\end{equation}%
where $f:\mathbb{R\longrightarrow }\mathbb{R}$ is a smooth compactly
supported function with $f(x)=x$ on $[-l,l]$ for some large $l>0$.

Hence, in summary a reasonable applicable stochastic volatility model in our
setting, which takes into account both volatility roughness and regime
switching effects, could be the following: 
\begin{eqnarray}
S_{t}^{x_{1},x_{2}} &=&x_{1}+\int_{0}^{t}\mu
S_{u}^{x_{1},x_{2}}du+\int_{0}^{t}g(\sigma _{u}^{x_{2}}/\nu
)S_{u}^{x_{1},x_{2}}dW_{u}  \notag \\
\sigma _{t}^{x_{2}} &=&x_{2}/\nu -\int_{0}^{t}b(u,\sigma
_{u}^{x_{2}})du+B_{t}^{H},x_{1},x_{2}\in \mathbb{R},0\leq t\leq T,
\label{Our}
\end{eqnarray}%
where $g$ is a smooth function of the form (\ref{Truncation}) and where the
vector field $b$ is given by (\ref{RS}).

Since the coefficient $b$ in (\ref{Our}) can be decomposed as a sum of
measurable bounded function and a Lipschitz function (of linear growth), one
can in fact prove the following BEL-representation:

\begin{theorem}
\label{Volatility2}Let $U\subset \mathbb{R}^{2}$ be a bounded, open. Suppose
that $b:[0,T]\times \mathbb{R}\longrightarrow \mathbb{R}$ in the stock price
model (\ref{Stock}) has the decomposition%
\begin{equation}
b=\widetilde{b}+\widehat{b},  \label{Decomp}
\end{equation}%
where $\widetilde{b}\in L^{\infty }([0,T]\times \mathbb{R})$ and where $%
\widehat{b}:[0,T]\times \mathbb{R}\longrightarrow \mathbb{R}$ satisfies a
linear growth and Lipschitz condition uniformly in time. Further, require
that $g:\mathbb{R}\longrightarrow (0,\infty )$ is given as in (\ref%
{Truncation}) and that $\Phi :\mathbb{R}^{2}\longrightarrow \mathbb{R}$ is a
Borel-measurable function such that%
\begin{equation*}
\Phi (S_{T}^{\cdot ,\cdot },\sigma _{T}^{\cdot })\in L^{2}(\Omega \times
U,\mu \times dx)\text{.}
\end{equation*}%
Let $a$ be a bounded and measurable function on $[0,T]$, which sums up to $1$%
. Then for $H<\frac{1}{2}$ and $\rho \in (-1,0)\cup (0,1)$ the function $%
u:U\longrightarrow \mathbb{R}$ defined by%
\begin{equation*}
u(x)=E[\Phi (S_{T}^{x_{1},x_{2}},\sigma _{T}^{x_{2}})],x\in U
\end{equation*}%
belongs to $C^{1}(U)$ and $\frac{\partial }{\partial x}u(x)$ has the
representation (\ref{SBEL}).
\end{theorem}

\bigskip

The proof of Theorem \ref{Volatility2} requires some notions and auxiliary
results.

Using the compactness criterion for square integrable functionals of Wiener
processes from Malliavin calculus \cite{DMN}, the proof of the next result
gives an alternative method for the construction of unique strong solutions
of (\ref{Pardoux}) to the work of \cite{NO} in the case of vector fields
given by (\ref{Decomp}).

\begin{theorem}
\label{Ouknine}Consider the SDE%
\begin{equation}
X_{t}^{x}=x+\int_{0}^{t}b(u,X_{u}^{x})du+\rho _{1}B_{t}^{H}+\rho
_{2}W_{t},x\in \mathbb{R},0\leq t\leq T,  \label{Pardoux}
\end{equation}%
where $B_{\cdot }^{H}$ is a fractional Brownian motion with $H<\frac{1}{2}$
being independent of a Wiener process $W_{\cdot }$, where $\rho _{1},\rho
_{2}\in \mathbb{R}\setminus \{0\}$. Assume for $b$ the decomposition (\ref%
{Decomp}) with respect to coefficients $\widetilde{b},\widehat{b}$
satisfying the conditions of Theorem \ref{Volatility2}. Then there exists a
unique strong solution $X_{\cdot }^{x}$ to the SDE (\ref{Pardoux}).
Moreover, $X_{t}^{x}$ is Malliavin differentiable in the direction of $%
(B_{\cdot },W_{\cdot })^{\ast }$ (and $(B_{\cdot }^{H},W_{\cdot })^{\ast }$)
for all $t$, where $B_{\cdot }$ is the Wiener process in the stochastic
integral representation of $B_{\cdot }^{H}$.
\end{theorem}

Let us now recall the concept of the local time-space integral, which goes
back to \cite{Eisenbaum} and which in the following form was given in \cite%
{BMPD}:

\begin{definition}
\label{LocalTime}Let $(\mathcal{H}^{x},\left\Vert \cdot \right\Vert )$ be
the Banach space of Borel measurable functions $f:[0,T]\times \mathbb{R}%
\longrightarrow $ $\mathbb{R}$ endowed with the norm $\left\Vert \cdot
\right\Vert _{x}$ given by%
\begin{eqnarray*}
\left\Vert f\right\Vert _{x} &:&=2(\int_{0}^{T}\int_{\mathbb{R}}f^{2}(s,y)%
\frac{1}{\sqrt{2\pi s}}\exp (-\frac{\left\vert y-x\right\vert ^{2}}{2s}%
)dyds)^{1/2} \\
&&+\int_{0}^{T}\int_{\mathbb{R}}\left\vert y-x\right\vert \left\vert
f(s,y)\right\vert \frac{1}{s\sqrt{2\pi s}}\exp (-\frac{\left\vert
y-x\right\vert ^{2}}{2s})dyds.
\end{eqnarray*}%
Denote by $f_{\Delta }:[0,T]\times \mathbb{R}\longrightarrow $ $\mathbb{R}$
a simple function of the form%
\begin{equation*}
f_{\Delta }(s,x)=\sum_{1\leq i\leq n-1,1\leq j\leq
m-1}f_{ij}1_{(y_{i},y_{i+1}]}(y)1_{(s_{j},s_{j+1}]}(s),
\end{equation*}%
where $(s_{j})_{1\leq j\leq m}$ is a partition of $[0,T]$ and where $%
(y_{i})_{1\leq i\leq n}$ and $(f_{ij})_{1\leq i\leq n,1\leq j\leq m}$ are
finite sequences of real numbers. Let $L^{X^{x}}(t,y)$ be the local time of
the solution $X_{\cdot }^{x}$ to (\ref{Pardoux}) for $\rho _{1}=0$ and $\rho
_{2}=1$ and vector fields $b$ as in (\ref{Decomp}). Then the local
time-space integral of a simple function $f_{\Delta }$ with respect to the
integrator $L^{X^{x}}(dt,dy)$ is defined by%
\begin{eqnarray*}
&&\int_{0}^{T}\int_{\mathbb{R}}f_{\Delta }(s,y)L^{X^{x}}(ds,dy) \\
&=&\sum_{1\leq i\leq n-1,1\leq j\leq
m-1}f_{ij}(L^{X^{x}}(s_{j+1},y_{i+1})-L^{X^{x}}(s_{j},y_{i+1})-L^{X^{x}}(s_{j+1},y_{i})+L^{X^{x}}(s_{j},y_{i})).
\end{eqnarray*}%
The class of simple functions is dense in $(\mathcal{H}^{x},\left\Vert \cdot
\right\Vert )$. For $f\in \mathcal{H}^{x}$ let $f_{n},n\geq 1$ be a sequence
of simple functions converging to $f$ in $\mathcal{H}^{x}$. Then the local
time-space integral of $f$ can be defined as the following (existing) limit
in probability:%
\begin{equation*}
\int_{0}^{T}\int_{\mathbb{R}}f(s,y)L^{X^{x}}(ds,dy):=\lim_{n\longrightarrow
\infty }\int_{0}^{T}\int_{\mathbb{R}}f_{n}(s,y)L^{X^{x}}(ds,dy).
\end{equation*}%
See Lemma 2.7 in \cite{BMPD}.
\end{definition}

\bigskip

In the sequel, we define%
\begin{equation}
\int_{0}^{t}\int_{\mathbb{R}}f(s,y)L^{X^{x}}(ds,dy):=\int_{0}^{T}\int_{%
\mathbb{R}}1_{[0,t]}(s)f(s,y)L^{X^{x}}(ds,dy)  \label{t}
\end{equation}%
for $0\leq t\leq T$, if $f\in \mathcal{H}^{x}$.

We also need the following representation of local time-space integrals (\ref%
{t}) in the case of $X_{t}^{x}=W_{t}^{x}:=x+W_{t}$, which is due to \cite%
{Eisenbaum}:

\begin{lemma}
\label{Eisenbaum}If $f\in \mathcal{H}^{0}$, then%
\begin{eqnarray}
&&\int_{0}^{t}\int_{\mathbb{R}}f(s,y)L^{W^{x}}(ds,dy)  \notag \\
&=&\int_{0}^{t}f(s,W_{s}^{x})dW_{s}+\int_{T-t}^{T}f(T-s,\widehat{W}%
_{s}^{x})dW_{s}^{\ast }-\int_{T-t}^{T}f(T-s,\widehat{W}_{s}^{x})\frac{%
\widehat{W}_{s}}{T-s}ds,  \label{Eisen}
\end{eqnarray}%
where $\widehat{W}_{t}:=W_{T-t},0\leq t\leq T$ is the time-reversed Wiener
process and 
\begin{equation*}
W_{t}^{\ast }:=\widehat{W}_{t}-W_{T}+\int_{0}^{t}\frac{\widehat{W}_{s}}{T-s}%
ds,0\leq t\leq T
\end{equation*}%
is a Wiener process with respect to the filtration of $\widehat{W}$.
\end{lemma}

\bigskip

Later on we will also make use of the following integration by parts
relation with respect to local time-space integrals (see \cite{Eisenbaum}, 
\cite{BMPD}):

\begin{lemma}
\label{IBP}Suppose $f\in \mathcal{H}^{x}$ is Lipschitz continuous with
respect to the spatial variable and denote by $f^{\shortmid }$ its spatial
weak derivative. Then all $0\leq t\leq T$ $X_{t}^{x}$ is Malliavin
differentiable and 
\begin{equation*}
-\int_{0}^{t}\int_{\mathbb{R}}f(s,y)L^{X^{x}}(ds,dy)=\int_{0}^{t}f^{%
\shortmid }(s,X_{s}^{x})ds\text{ a.e.}
\end{equation*}
\end{lemma}

Using mollification let us now in view of the next auxiliary result consider
smooth functions $\widehat{b}_{n},n\geq 1$ such that

\begin{equation}
\widehat{b}_{n}(t,x)\longrightarrow \widehat{b}(t,x),\widehat{b}%
_{n}^{\shortmid }(t,x)\longrightarrow \widehat{b}^{\shortmid }(t,x),\text{ }%
(t,x)-\text{a.e.,}  \label{b1}
\end{equation}%
\begin{equation*}
\left\vert \widehat{b}_{n}(t,x)\right\vert \leq C(1+\left\vert x\right\vert
),x\in \mathbb{R},0\leq t\leq T,n\geq 1
\end{equation*}

\bigskip as well as

\begin{equation*}
\left\vert \widehat{b}_{n}(t,x)-\widehat{b}_{n}(t,y)\right\vert \leq
K\left\vert x-y\right\vert ,x,y\in \mathbb{R},0\leq t\leq T,n\geq 1
\end{equation*}%
where $C>0$ and $K$ is the Lipschitz constant of $\widehat{b}$. So%
\begin{equation*}
\left\vert \widehat{b}_{n}^{\shortmid }(t,x)\right\vert \leq K
\end{equation*}%
for all $x\in \mathbb{R},0\leq t\leq T,n\geq 1$. Further, let $\widetilde{b}%
_{n},n\geq 1$ be a sequence of smooth and compactly supported functions such
that 
\begin{subequations}
\begin{equation}
\widetilde{b}_{n}(t,x)\underset{n\longrightarrow \infty }{\longrightarrow }%
\widetilde{b}(t,x)\text{ }(t,x)-\text{a.e.}  \label{b2}
\end{equation}%
and 
\end{subequations}
\begin{equation*}
\left\vert \widetilde{b}_{n}(t,x)\right\vert \leq L
\end{equation*}%
for all $t,x$ for some constant $L<\infty $.

\begin{lemma}
\label{WienerTransform} Let $\widetilde{b}_{n},\widehat{b}_{n},n\geq 1$ be
as in (\ref{b1}), (\ref{b2}) and $b_{n}=\widetilde{b}_{n}+\widehat{b}%
_{n},n\geq 1$. Assume that $X_{\cdot }^{x,n}$ is the strong solution
associated with the vector field $b_{n},n\geq 1$. Further, let $X_{\cdot
}^{x}$ be the weak solution to (\ref{Pardoux}) and $\left\{ \mathcal{F}%
_{t}\right\} _{0\leq t\leq T}$ the ($\mu -$completed) filtration generated
by $B_{t}^{H},W_{t},0\leq t\leq T$. Then, for all $0\leq t\leq T$%
\begin{equation*}
X_{t}^{x,n}\underset{n\longrightarrow \infty }{\longrightarrow }E\left[
X_{t}^{x}\right. \left\vert \mathcal{F}_{t}\right]
\end{equation*}%
as well as%
\begin{equation*}
(X_{t}^{x,n})^{2}\underset{n\longrightarrow \infty }{\longrightarrow }E\left[
(X_{t}^{x})^{2}\right. \left\vert \mathcal{F}_{t}\right]
\end{equation*}%
weakly in $L^{2}(\Omega ,\mathcal{F}_{t})$.
\end{lemma}

\begin{proof}
The proof is the same as that of Lemma A.3 in \cite{BMPD}.
\end{proof}

\begin{proof}[Proof of Theorem \protect\ref{Ouknine}]
Assume for notational convenience that $\rho _{1}$\bigskip $=\rho _{2}=1$.
Let $\widehat{b}_{n},\widetilde{b}_{n},n\geq 1$ the smooth approximating
sequences of functions in (\ref{b1}) and (\ref{b2}).

In proving this result, we aim at applying as in Section 2 the compactness
criterion in \cite{DMN} to the sequence $X_{t}^{x,n},n\geq 1$ for each fixed 
$t$. Just as in \cite{BMPD} we can show for $b_{n}:=\widetilde{b}_{n}+%
\widehat{b}_{n}$ that 
\begin{eqnarray}
&&D_{s}^{W}X_{t}^{x,n}  \label{DW} \\
&=&\exp \{\int_{s}^{t}\widetilde{b}_{n}^{\prime }(u,X_{u}^{x,n})+\widehat{b}%
_{n}^{\prime }(u,X_{u}^{x,n})du\}\text{ }\mu -a.e,s\leq t\text{ }ds-a.e, 
\notag
\end{eqnarray}%
where $\widehat{b}^{\prime }$ is the weak spatial derivative. Further, by
using Girsanov's theorem with respect $W_{\cdot }$, the mean value theorem
and H\"{o}lder's inequality we obtain similarly to the proof of Theorem A.4
in \cite{BMPD} that for $0\leq s\leq s^{\prime }\leq t$ and $x\in K\subset 
\mathbb{R}$ ($K$ compact)%
\begin{eqnarray*}
&&\left\Vert D_{s}^{W}X_{t}^{x,n}-D_{s^{\prime }}^{W}X_{t}^{x,n}\right\Vert
_{L^{2}(\mu )}^{2} \\
&\leq &E[\exp \{2\frac{1+\varepsilon }{\varepsilon }\int_{s^{\prime
}}^{t}b_{n}^{\prime }(u,x+B_{u}^{H}+W_{u})du\} \\
&&\times \sup_{0\leq \theta \leq 1}\exp \{2\frac{1+\varepsilon }{\varepsilon 
}\theta \int_{s}^{s^{\prime }}b_{n}^{\prime }(u,x+B_{u}^{H}+W_{u})du\} \\
&&\times \left\vert \int_{s}^{s^{\prime }}b_{n}^{\prime
}(u,x+B_{u}^{H}+W_{u})du\right\vert ^{2\frac{1+\varepsilon }{\varepsilon }%
}]^{\frac{1+\varepsilon }{\varepsilon }}E[\mathcal{E}(b_{n})_{T}^{1+%
\varepsilon }]^{\frac{1}{1+\varepsilon }},
\end{eqnarray*}%
where%
\begin{eqnarray*}
&&\mathcal{E}(b)_{T} \\
&:&=\exp (\int_{0}^{T}b(u,x+B_{u}^{H}+W_{u})dW_{u}-\frac{1}{2}%
\int_{0}^{T}(b(u,x+B_{u}^{H}+W_{u}))^{2}du)
\end{eqnarray*}%
and where $\varepsilon $ is chosen such that 
\begin{equation*}
\sup_{x\in K}\sup_{n\geq 1}E[\mathcal{E}(b_{n})_{T}^{1+\varepsilon }]<\infty 
\text{.}
\end{equation*}%
Using the latter estimate combined with the same arguments as in Theorem A.4
in \cite{BMPD}, we find that%
\begin{eqnarray*}
&&\left\Vert D_{s}^{W}X_{t}^{x,n}-D_{s^{\prime }}^{W}X_{t}^{x,n}\right\Vert
_{L^{2}(\mu )}^{2} \\
&\leq &C(\left\Vert \widehat{b}_{n}^{\shortmid }\right\Vert _{\infty
}^{2}T\left\vert s-s^{\prime }\right\vert +\left\Vert \widetilde{b}%
_{n}\right\Vert _{\infty }^{2}\left\vert s-s^{\prime }\right\vert \\
&&+\left\Vert \widetilde{b}_{n}\right\Vert _{\infty }^{2}\left(
\int_{T-s^{\prime }}^{T-s}\left\Vert \frac{B_{T-u}}{T-u}\right\Vert _{L^{2(4%
\frac{1+\varepsilon }{\varepsilon })}(\mu )}du\right) ^{2}) \\
&\leq &C(\left\Vert \widehat{b}_{n}^{\shortmid }\right\Vert _{\infty
}^{2}T\left\vert s-s^{\prime }\right\vert +\left\Vert \widetilde{b}%
_{n}\right\Vert _{\infty }^{2}\left\vert s-s^{\prime }\right\vert +%
\widetilde{C}\left\Vert \widetilde{b}_{n}\right\Vert _{\infty
}^{2}\left\vert s-s^{\prime }\right\vert ).
\end{eqnarray*}%
So%
\begin{equation}
\left\Vert D_{s}^{W}X_{t}^{x,n}-D_{s^{\prime }}^{W}X_{t}^{x,n}\right\Vert
_{L^{2}(\mu )}^{2}\leq C\left\vert s-s^{\prime }\right\vert  \label{Meq}
\end{equation}%
for all $0\leq s\leq s^{\prime }\leq t$, where $C$ is a constant depending
on $H,T,\varepsilon $ and $b$. Therefore, there exists a $\beta \in (0,\frac{%
1}{2})$ and a constant $C<\infty $ depending on $H,T$,$\varepsilon $ and $b$
such that for all $0\leq s\leq s^{\prime }\leq t$: 
\begin{equation}
\sup_{x\in K}\sup_{n\geq 1}\int_{0}^{T}\int_{0}^{T}\frac{\left\Vert
D_{s}^{W}X_{t}^{x,n}-D_{s^{\prime }}^{W}X_{t}^{x,n}\right\Vert _{L^{2}(\mu
)}^{2}}{\left\vert s^{\prime }-s\right\vert ^{1+\alpha }}dsds^{\prime }\leq C%
\text{.}  \label{DW1}
\end{equation}%
In the same way, we can also verify that%
\begin{equation}
\sup_{x\in K}\sup_{n\geq 1}\left\Vert D_{\cdot }^{W}X_{t}^{x,n}\right\Vert
_{L^{2}([0,t]\times \Omega )}^{2}\leq C  \label{DW2}
\end{equation}%
for a constant $C<\infty $.

In order to employ the compactness criterion in \cite{DMN}, we next also
have to prove the estimates (\ref{DW2}), (\ref{DW1}) for the Malliavin
derivative $D^{B}$ in the direction of the Wiener process $B_{\cdot }$ in
the stochastic integral representation of $B_{\cdot }^{H}$. To this end, we
can use the transfer principle for Malliavin derivatives of Proposition
5.2.1 in \cite{Nualart} and get the following representation%
\begin{eqnarray*}
&&D_{s}^{B}X_{t}^{x,n} \\
&=&K_{H}(t,s)D_{s}^{H}X_{t}^{x,n} \\
&&+c_{H}(H-\frac{1}{2}%
)\int_{s}^{T}(D_{u}^{H}X_{t}^{x,n}-D_{s}^{H}X_{t}^{x,n})\left( \frac{u}{s}%
\right) ^{H-\frac{1}{2}}\frac{1}{(u-s)^{\frac{3}{2}-H}}du
\end{eqnarray*}%
$\mu -a.e,s\in \lbrack 0,t]$ $a.e.$ Note here that%
\begin{equation}
\frac{\partial }{\partial t}K(t,s)=c_{H}(H-\frac{1}{2})\left( \frac{t}{s}%
\right) ^{H-\frac{1}{2}}\frac{1}{(t-s)^{\frac{3}{2}-H}},t>s.  \label{K}
\end{equation}%
So%
\begin{eqnarray*}
&&\left\Vert D_{s}^{B}X_{t}^{x,n}\right\Vert _{L^{2}(\mu )} \\
&\leq &K_{H}(t,s)\left\Vert D_{s}^{H}X_{t}^{x,n}\right\Vert _{L^{2}(\mu )} \\
&&+c_{H}(\frac{1}{2}-H)\int_{s}^{T}\left\Vert
D_{u}^{H}X_{t}^{x,n}-D_{s}^{H}X_{t}^{x,n}\right\Vert _{L^{2}(\mu )}\left( 
\frac{u}{s}\right) ^{H-\frac{1}{2}}\frac{1}{(u-s)^{\frac{3}{2}-H}}du.
\end{eqnarray*}%
On the other hand, by using the chain rule for the Malliavin derivative we
obtain just as in (\ref{DW}) the representation 
\begin{eqnarray*}
&&D_{s}^{H}X_{t}^{x,n} \\
&=&\exp \{\int_{s}^{t}\widetilde{b}_{n}^{\prime }(u,X_{u}^{x,n})+\widehat{b}%
_{n}^{\prime }(u,X_{u}^{x,n})du\}\text{ }\mu -a.e,s\leq t\text{ }ds-a.e.
\end{eqnarray*}

So as in the case of $D_{\cdot }^{W}$ we can get the estimate%
\begin{equation}
\left\Vert D_{s}^{H}X_{t}^{x,n}-D_{s^{\prime }}^{H}X_{t}^{x,n}\right\Vert
_{L^{2}(\mu )}^{2}\leq C\left\vert s-s^{\prime }\right\vert
\end{equation}%
for all $0\leq s\leq s^{\prime }\leq t$, where $C$ is a constant depending
on $H,T,\varepsilon $ and $b$.

Thus for all compact sets $K\subset \mathbb{R}$ there is a constant $C$
depending on $T,K,\varepsilon $ and $b$ such that%
\begin{equation}
\sup_{x\in K}\sup_{0\leq u\leq t}\sup_{n\geq 1}\left\Vert
D_{u}^{H}X_{t}^{x,n}\right\Vert _{L^{2}(\mu )}\leq C  \label{Est1}
\end{equation}%
as well as%
\begin{equation}
\sup_{x\in K}\sup_{n\geq 1}\left\Vert
D_{u}^{H}X_{t}^{x,n}-D_{s}^{H}X_{t}^{x,n}\right\Vert _{L^{2}(\mu )}\leq
C\left\vert u-s\right\vert ^{\frac{1}{2}}  \label{Est2}
\end{equation}%
for all $0\leq s\leq u\leq t$. Hence,%
\begin{eqnarray*}
&&\left\Vert D_{s}^{B}X_{t}^{x,n}\right\Vert _{L^{2}(\mu )} \\
&\leq &C(K_{H}(t,s)+c_{H}(\frac{1}{2}-H)\int_{s}^{T}\left( \frac{u}{s}%
\right) ^{H-\frac{1}{2}}\frac{1}{(u-s)^{1-H}}du) \\
&\leq &C(K_{H}(t,s)+c_{H}(\frac{1}{2}-H)\int_{s}^{T}\left( \frac{s}{s}%
\right) ^{H-\frac{1}{2}}\frac{1}{(u-s)^{1-H}}du) \\
&=&C(K_{H}(t,s)+c_{H}(\frac{1}{2}-H)\frac{1}{H}(T-s)^{H}).
\end{eqnarray*}%
The latter entails that%
\begin{equation*}
\sup_{x\in K}\sup_{n\geq 1}\left\Vert D_{\cdot }^{B}X_{t}^{x,n}\right\Vert
_{L^{2}(\mu \times \left[ 0,T\right] )}\leq C^{\ast }
\end{equation*}%
for a constant $C^{\ast }=C^{\ast }(H,T,\varepsilon ,b)>0$. Further, we get
for $s_{2}\geq s_{1}$ that

\begin{eqnarray*}
&&D_{s_{2}}^{B}X_{t}^{x,n}-D_{s_{1}}^{B}X_{t}^{x,n} \\
&=&K_{H}(t,s_{2})D_{s_{2}}^{H}X_{t}^{x,n}-K_{H}(t,s_{1})D_{s_{1}}^{H}X_{t}^{x,n}
\\
&&+c_{H}(H-\frac{1}{2}%
)\int_{s_{2}}^{T}(D_{u}^{H}X_{t}^{x,n}-D_{s_{2}}^{H}X_{t}^{x,n})\left( \frac{%
u}{s_{2}}\right) ^{H-\frac{1}{2}}\frac{1}{(u-s_{2})^{\frac{3}{2}-H}}du \\
&&-c_{H}(H-\frac{1}{2}%
)\int_{s_{1}}^{T}(D_{u}^{H}X_{t}^{x,n}-D_{s_{1}}^{H}X_{t}^{x,n})\left( \frac{%
u}{s_{1}}\right) ^{H-\frac{1}{2}}\frac{1}{(u-s_{1})^{\frac{3}{2}-H}}du \\
&=&K_{H}(t,s_{2})(D_{s_{2}}^{H}X_{t}^{x,n}-D_{s_{1}}^{H}X_{t}^{x,n})+(K_{H}(t,s_{2})-K_{H}(t,s_{1}))D_{s_{1}}^{H}X_{t}^{x,n}
\\
&&+\int_{s_{2}}^{T}(D_{u}^{H}X_{t}^{x,n}-D_{s_{2}}^{H}X_{t}^{x,n})(\frac{%
\partial }{\partial u}K(u,s_{2})-\frac{\partial }{\partial u}K(u,s_{1}))du \\
&&-c_{H}(H-\frac{1}{2}%
)\int_{s_{2}}^{T}(D_{s_{2}}^{H}X_{t}^{x,n}-D_{s_{1}}^{H}X_{t}^{x,n})\left( 
\frac{u}{s_{1}}\right) ^{H-\frac{1}{2}}\frac{1}{(u-s_{1})^{\frac{3}{2}-H}}du
\\
&&-c_{H}(H-\frac{1}{2}%
)\int_{s_{1}}^{s_{2}}(D_{u}^{H}X_{t}^{x,n}-D_{s_{1}}^{H}X_{t}^{x,n})\left( 
\frac{u}{s_{1}}\right) ^{H-\frac{1}{2}}\frac{1}{(u-s_{1})^{\frac{3}{2}-H}}du
\\
&=&\sum_{j=1}^{5}I_{j}(s_{1},s_{2}).
\end{eqnarray*}

Let us first have a look at the most difficult term, that is%
\begin{equation*}
I_{3}(s_{1},s_{2})=%
\int_{s_{2}}^{T}(D_{u}^{H}X_{t}^{x,n}-D_{s_{2}}^{H}X_{t}^{x,n})\frac{%
\partial }{\partial u}(K_{H}(u,s_{2})-K_{H}(u,s_{1}))du
\end{equation*}%
In what follows, we aim at using the following inequality (see the proof of
Lemma A.4 in \cite{BNP}):

For all $\gamma \in (0,H),0<\theta _{1}<\theta _{2}<T$ we have that

\begin{equation}
\left\vert K_{H}(t,\theta _{1})-K_{H}(t,\theta _{2})\right\vert \leq C_{H,T}%
\frac{(\theta _{2}-\theta _{1})^{\gamma }}{(\theta _{2}\theta _{1})^{\gamma }%
}\theta _{2}^{H-\frac{1}{2}-\gamma }(t-\theta _{2})^{H-\frac{1}{2}-\gamma }.
\label{Keq}
\end{equation}%
It is shown in \cite{BNP} (Proof of Lemma A.4)) that there exists a $\beta
>0 $ depending on $H$ and $\gamma $ such that 
\begin{equation}
\int_{0}^{t}\int_{0}^{t}\left( \frac{\left\vert \theta _{2}-\theta
_{1}\right\vert ^{2\gamma }}{(\theta _{2}\theta _{1})^{2\gamma }}\theta
_{2}^{2H-1-2\gamma }(t-\theta _{2})^{2H-1-2\gamma }\right) \left\vert \theta
_{2}-\theta _{1}\right\vert ^{-1+\beta }d\theta _{1}d\theta _{2}<\infty
\label{FracKeq}
\end{equation}

Since 
\begin{equation*}
\left\vert K(t,s)\right\vert \leq C(t-s)^{H-\frac{1}{2}},t>s
\end{equation*}%
for a constant $C$ and since $D_{u}^{H}X_{t}^{x,n}$ is Lipschitz continuous
in $u$, we obtain from the definition of $K_{H}$, integration by parts and
inequality (\ref{Keq}) that%
\begin{eqnarray*}
&&\left\Vert I_{3}(s_{1},s_{2})\right\Vert _{L^{2}(\mu )} \\
&\leq &c_{H}(\frac{1}{2}-H)\int_{s_{2}}^{T}\left\Vert
D_{u}^{H}X_{t}^{x,n}-D_{s_{2}}^{H}X_{t}^{x,n}\right\Vert _{L^{2}(\mu
)}(\left( \frac{s_{2}}{u}\right) ^{\frac{1}{2}-H}\frac{1}{(u-s_{2})^{\frac{3%
}{2}-H}}-\left( \frac{s_{1}}{u}\right) ^{\frac{1}{2}-H}\frac{1}{(u-s_{1})^{%
\frac{3}{2}-H}})du \\
&\leq &c_{H}(\frac{1}{2}-H)\int_{s_{2}}^{T}\left\Vert
D_{u}^{H}X_{t}^{x,n}-D_{s_{2}}^{H}X_{t}^{x,n}\right\Vert _{L^{2}(\mu )}\frac{%
1}{u^{\frac{1}{2}-H}}(s_{2}^{\frac{1}{2}-H}\frac{1}{(u-s_{2})^{\frac{3}{2}-H}%
}-s_{1}^{\frac{1}{2}-H}\frac{1}{(u-s_{1})^{\frac{3}{2}-H}})du \\
&\leq &c_{H}(\frac{1}{2}-H)\int_{s_{2}}^{T}C(u-s_{2})^{\frac{1}{2}}\frac{1}{%
u^{\frac{1}{2}-H}}(s_{2}^{\frac{1}{2}-H}\frac{1}{(u-s_{2})^{\frac{3}{2}-H}}%
-s_{1}^{\frac{1}{2}-H}\frac{1}{(u-s_{1})^{\frac{3}{2}-H}})du \\
&=&\lim_{s\searrow s_{2}}\left. C(u-s_{2})^{\frac{1}{2}%
}(K(u,s_{2})-K(u,s_{1}))\right\vert _{u=s}^{T}-\int_{s_{2}}^{T}\frac{C}{2}%
(u-s_{2})^{-\frac{1}{2}}(K(u,s_{2})-K(u,s_{1}))du \\
&=&C(T-s_{2})^{\frac{1}{2}}(K(T,s_{2})-K(T,s_{1}))--\int_{s_{2}}^{T}\frac{C}{%
2}(u-s_{2})^{-\frac{1}{2}}(K(u,s_{2})-K(u,s_{1}))du \\
&\leq &C(T-s_{2})^{\frac{1}{2}}C_{H,T}\frac{(s_{2}-s_{1})^{\gamma }}{%
(s_{2}s_{1})^{\gamma }}s_{2}^{H-\frac{1}{2}-\gamma }(T-s_{2})^{H-\frac{1}{2}%
-\gamma } \\
&&+\int_{s_{2}}^{T}\frac{C}{2}(u-s_{2})^{-\frac{1}{2}}C_{H,T}\frac{%
(s_{2}-s_{1})^{\gamma }}{(s_{2}s_{1})^{\gamma }}s_{2}^{H-\frac{1}{2}-\gamma
}(u-s_{2})^{H-\frac{1}{2}-\gamma }du \\
&=&C(T-s_{2})^{\frac{1}{2}}C_{H,T}\frac{(s_{2}-s_{1})^{\gamma }}{%
(s_{2}s_{1})^{\gamma }}s_{2}^{H-\frac{1}{2}-\gamma }(T-s_{2})^{H-\frac{1}{2}%
-\gamma } \\
&&+\int_{s_{2}}^{T}\frac{C}{2}C_{H,T}\frac{(s_{2}-s_{1})^{\gamma }}{%
(s_{2}s_{1})^{\gamma }}s_{2}^{H-\frac{1}{2}-\gamma }(u-s_{2})^{H-1-\gamma }du
\\
&=&CC_{H,T}\frac{(s_{2}-s_{1})^{\gamma }}{(s_{2}s_{1})^{\gamma }}s_{2}^{H-%
\frac{1}{2}-\gamma }(T-s_{2})^{H-\gamma } \\
&&+\frac{C}{2}C_{H,T}\frac{(s_{2}-s_{1})^{\gamma }}{(s_{2}s_{1})^{\gamma }}%
s_{2}^{H-\frac{1}{2}-\gamma }\frac{1}{H-1-\gamma }(T-s_{2})^{H-\gamma }
\end{eqnarray*}

Hence, we conclude from (\ref{FracKeq}) that there exists a $\beta >0$
depending on $H$ and $\gamma $ such that%
\begin{equation*}
\int_{0}^{T}\int_{0}^{T}\frac{\left\Vert I_{3}(s_{1},s_{2})\right\Vert
_{L^{2}(\mu )}^{2}}{\left\vert s_{2}-s_{1}\right\vert ^{1+2\beta }}%
ds_{1}ds_{2}<\infty \text{.}
\end{equation*}%
Further, we see that%
\begin{eqnarray*}
&&\left\Vert I_{4}(s_{1},s_{2})\right\Vert _{L^{2}(\mu )} \\
&\leq &c_{H}(\frac{1}{2}-H)\int_{s_{2}}^{T}\left\Vert
D_{s_{2}}^{H}X_{t}^{x,n}-D_{s_{1}}^{H}X_{t}^{x,n}\right\Vert _{L^{2}(\mu
)}\left( \frac{u}{s_{1}}\right) ^{H-\frac{1}{2}}\frac{1}{(u-s_{1})^{\frac{3}{%
2}-H}}du \\
&\leq &c_{H}(\frac{1}{2}-H)C(s_{2}-s_{1})^{\frac{1}{2}}\int_{s_{2}}^{T}\frac{%
1}{(u-s_{1})^{\frac{3}{2}-H}}du \\
&=&c_{H}(\frac{1}{2}-H)C(s_{2}-s_{1})^{\frac{1}{2}}\frac{1}{H-3/2}(\frac{1}{%
(T-s_{1})^{\frac{1}{2}-H}}-\frac{1}{(s_{2}-s_{1})^{\frac{1}{2}-H}}) \\
&=&c_{H}(\frac{1}{2}-H)C\frac{1}{H-3/2}((s_{2}-s_{1})^{\frac{1}{2}}\frac{1}{%
(T-s_{1})^{\frac{1}{2}-H}}+(s_{2}-s_{1})^{H}) \\
&\leq &2c_{H}(\frac{1}{2}-H)C\frac{1}{H-3/2}(s_{2}-s_{1})^{H}.
\end{eqnarray*}%
So%
\begin{equation*}
\int_{0}^{T}\int_{0}^{T}\frac{\left\Vert I_{4}(s_{1},s_{2})\right\Vert
_{L^{2}(\mu )}^{2}}{\left\vert s_{2}-s_{1}\right\vert ^{1+2\beta }}%
ds_{1}ds_{2}<\infty
\end{equation*}%
for a $\beta >0$ depending on $H$.

In addition, it also follows that%
\begin{eqnarray*}
&&\left\Vert I_{5}(s_{1},s_{2})\right\Vert _{L^{2}(\mu )} \\
&\leq &c_{H}(\frac{1}{2}-H)\int_{s_{1}}^{s_{2}}\left\Vert
D_{u}^{H}X_{t}^{x,n}-D_{s_{1}}^{H}X_{t}^{x,n}\right\Vert _{L^{2}(\mu
)}\left( \frac{u}{s_{1}}\right) ^{H-\frac{1}{2}}\frac{1}{(u-s_{1})^{\frac{3}{%
2}-H}}du \\
&\leq &c_{H}(\frac{1}{2}-H)C\int_{s_{1}}^{s_{2}}(u-s_{1})^{\frac{1}{2}}\frac{%
1}{(u-s_{1})^{\frac{3}{2}-H}}du \\
&=&c_{H}(\frac{1}{2}-H)C\frac{1}{1-H}(s_{2}-s_{1})^{H}.
\end{eqnarray*}%
Hence,%
\begin{equation*}
\int_{0}^{T}\int_{0}^{T}\frac{\left\Vert I_{5}(s_{1},s_{2})\right\Vert
_{L^{2}(\mu )}^{2}}{\left\vert s_{2}-s_{1}\right\vert ^{1+2\beta }}%
ds_{1}ds_{2}<\infty
\end{equation*}%
for a $\beta >0$ depending on $H$.

By using (\ref{Est1}), (\ref{Est2}), (\ref{Keq}) and (\ref{FracKeq}) we can
treat the terms $I_{1}(s_{1},s_{2})$ and $I_{2}(s_{1},s_{2})$ similarly and
we get altogether for compact sets $K\subset \mathbb{R}$ that 
\begin{equation*}
\sup_{x\in K}\sup_{n\geq 1}\int_{0}^{T}\int_{0}^{T}\frac{\left\Vert
D_{s}^{B}X_{t}^{x,n}-D_{s^{\prime }}^{B}X_{t}^{x,n}\right\Vert _{L^{2}(\mu
)}^{2}}{\left\vert s^{\prime }-s\right\vert ^{1+\alpha }}dsds^{\prime }\leq
C<\infty \text{.}
\end{equation*}%
Using the above estimates with respect to $D_{\cdot }^{B}$ and $D_{\cdot
}^{W}$ we can now apply the compactness criterion in \cite{DMN} and obtain
that for all $0\leq t\leq T$ there exists a subsequence $n_{k},k\geq 1$
depending on $t$ and $Y_{t}^{x}\in L^{2}(\Omega ,\mathcal{F}_{t})$%
\begin{equation*}
X_{t}^{x,n_{k}}\underset{k\longrightarrow \infty }{\longrightarrow }Y_{t}^{x}
\end{equation*}%
in $L^{2}(\Omega ,\mathcal{F}_{t})$. However, it follows from Lemma \ref%
{WienerTransform} that $Y_{t}^{x}=E\left[ X_{t}^{x}\right. \left\vert 
\mathcal{F}_{t}\right] $ a.e. Hence,%
\begin{equation}
X_{t}^{x,n}\underset{n\longrightarrow \infty }{\longrightarrow }E\left[
X_{t}^{x}\right. \left\vert \mathcal{F}_{t}\right]   \label{Lto}
\end{equation}%
in $L^{2}(\Omega ,\mathcal{F}_{t})$. The latter implies in connection with
Lemma \ref{WienerTransform} that%
\begin{equation*}
(X_{t}^{x})^{2}=E\left[ (X_{t}^{x})^{2}\right. \left\vert \mathcal{F}_{t}%
\right] \text{ a.e.}
\end{equation*}%
So $X_{t}^{x}=E\left[ X_{t}^{x}\right. \left\vert \mathcal{F}_{t}\right] $
a.e., which shows that the weak solution must be adapted to the filtration $%
\left\{ \mathcal{F}_{t}\right\} _{0\leq t\leq T}$ and therefore a strong
solution. Strong uniqueness is a consequence of Girsanov%
%TCIMACRO{\U{b4}}%
%BeginExpansion
\'{}%
%EndExpansion
s theorem. Further, the Malliavin differentiability of $X_{t}^{x}$ in the
direction of $(B_{\cdot },W_{\cdot })^{\ast }$, buyt also $(B_{\cdot
}^{H},W_{\cdot })^{\ast }$ follows from Lemma 1.2.3 in \cite{Nualart} in
connection with (\ref{Lto}) and the uniform estimates with respec to the
Malliavin derivatives of $X_{t}^{x,n}$.
\end{proof}

\bigskip

\begin{lemma}
Recall that $\Omega =\Omega _{1}\times \Omega _{2}$, where $B_{\cdot }^{H}$
is defined on $\Omega _{2}$ and $W_{\cdot }$ on $\Omega _{2}$. Let $X_{\cdot
}^{x}$ be the Malliavin differentiable solution of Theorem \ref{Ouknine}.
Then for all $0\leq t\leq T:$ 
\begin{eqnarray*}
&&D_{s}^{W}X_{t} \\
&=&\rho _{2}\exp \{-\int_{s}^{t}\int_{\mathbb{R}}f(r,y)L^{\rho
_{2}^{-1}(X^{x}-x-\rho _{1}B^{H}(\omega _{2}))}(\omega _{1},dr,dy)\}\text{ }%
\omega _{2}-a.e,\omega _{1}-a.e,s\leq t\text{ }ds-a.e
\end{eqnarray*}%
as well as%
\begin{eqnarray*}
&&D_{s}^{B}X_{t} \\
&=&\rho _{1}K_{H}(t,s)\exp \left\{ -\int_{s}^{t}\int_{\mathbb{R}%
}f(r,y)L^{\rho _{2}^{-1}(X^{x}-x-\rho _{1}B^{H}(\omega _{2}))}(\omega
_{1},dr,dy)\right\}  \\
&&+\rho _{1}c_{H}(\frac{1}{2}-H)\int_{s}^{T}(\exp \left\{ -\int_{u}^{t}\int_{%
\mathbb{R}}f(r,y)L^{\rho _{2}^{-1}(X^{x}-x-\rho _{1}B^{H}(\omega
_{2}))}(\omega _{1},dr,dy)\right\}  \\
&&\exp \left\{ -\int_{s}^{t}\int_{\mathbb{R}}f(r,y)L^{\rho
_{2}^{-1}(X^{x}-x-\rho _{1}B^{H}(\omega _{2}))}(\omega _{1},dr,dy)\right\}
)\left( \frac{u}{s}\right) ^{H-\frac{1}{2}}\frac{1}{(u-s)^{\frac{3}{2}-H}}du,
\end{eqnarray*}%
$\omega _{2}-a.e,\omega _{1}-a.e,s\leq t$ $ds-a.e$, where%
\begin{equation*}
f(s,y)=\rho _{2}^{-1}b(s,x+\rho _{1}B_{s}^{H}(\omega _{1})+\rho _{2}y).
\end{equation*}
\end{lemma}

\begin{proof}
\bigskip The first representation is a consequence of relation (\ref{DW}) in
connection with Lemma \ref{IBP} and Theorem \ref{Ouknine}, if we use the
sample space splitting $\Omega =\Omega _{1}\times \Omega _{2}$ for sample
spaces $\Omega _{1}$, $\Omega _{2}$, on which $W_{\cdot \text{ }}$ and $%
B_{\cdot }^{H}$ are defined, respectively. The second assertion follows from
the above mentioned transfer principle for Malliavin derivatives.
\end{proof}

\begin{lemma}
\label{RepDer}Retain the conditions and notation of Theorem \ref{Ouknine}.
Then for all $0\leq t\leq T$%
\begin{equation*}
X_{t}^{\cdot }\in L^{2}(\Omega ;W_{loc}^{1,2}(\mathbb{R}))
\end{equation*}%
and%
\begin{eqnarray*}
&&\frac{\partial }{\partial x}X_{t}^{x} \\
&=&\exp \left\{ -\int_{s}^{t}\int_{\mathbb{R}}f(s,y)L^{\rho
_{2}^{-1}(X^{x}-x-\rho _{1}B^{H}(\omega _{2}))}(\omega _{1},dr,dy)\right\} 
\text{ }\omega _{2}-a.e,\omega _{1}-a.e,x-a.e,
\end{eqnarray*}%
where%
\begin{equation*}
f(s,y)=\rho _{2}^{-1}b(s,x+\rho _{1}B_{s}^{H}(\omega _{1})+\rho _{2}y).
\end{equation*}
\end{lemma}

\begin{proof}
The proof is a direct consequence of the proof of Proposition 3.5 in \cite%
{BMPD}, if one uses the sample space splitting $\Omega =\Omega _{1}\times
\Omega _{2}$.
\end{proof}

\bigskip

\begin{lemma}
\label{u(x)}Adopt the conditions and notation of Theorem \ref{Ouknine}. Let $%
\Phi \in C_{0}^{\infty }(\mathbb{R})$ and define the functions $u_{n},$ $u$
for $T>0$ given by%
\begin{equation*}
u_{n}(x):=E[\Phi (X_{T}^{x,n})]\text{ and }u(x):=E[\Phi (X_{T}^{x})],
\end{equation*}%
where $X_{\cdot }^{x,n},n\geq 1$ is the approximating sequence of solutions
associated with the vector fields $b_{n},n\geq 1$ given by (\ref{b1}) and (%
\ref{b2}). Further, let $\overline{u}$ be the function defined as%
\begin{equation*}
\overline{u}(x):=E[\Phi ^{\shortmid }(X_{T}^{x})\frac{\partial }{\partial x}%
X_{T}^{x}].
\end{equation*}%
Then%
\begin{equation*}
u_{n}(x)\underset{n\longrightarrow \infty }{\longrightarrow }u(x)\text{ for
all }x
\end{equation*}%
as well as%
\begin{equation*}
u_{n}^{\shortmid }(x)\underset{n\longrightarrow \infty }{\longrightarrow }%
\overline{u}(x)
\end{equation*}%
uniformly on compact subsets $K\subset \mathbb{R}.$So $u\in C^{1}(\mathbb{R}%
) $ with $u^{\shortmid }=\overline{u}(x)$.
\end{lemma}

\begin{proof}
\bigskip The proof is the same as that of Lemma 4.1 in \cite{BMPD} applied
to the splitting $\Omega =\Omega _{1}\times \Omega _{2}$.
\end{proof}

\bigskip We are coming now to the proof of Theorem \ref{Volatility2}:

\begin{proof}[Proof of Theorem \protect\ref{Volatility2}]
: The proof is a consequence of that of Theorem \ref{Volatility} in
combination with Theorem \ref{Ouknine} and Lemma \ref{u(x)}.
\end{proof}

\bigskip

\begin{proposition}
\label{Stability}\bigskip \bigskip Let $H<\frac{1}{6}$ and let $X_{\cdot
}^{s,x}$ be the unique strong solution to the SDE%
\begin{equation}
dX_{t}^{s,x}=b(t,X_{t}^{s,x})dt+\rho _{1}dB_{t}^{H}+\rho
_{2}dW_{t}dB_{t}^{H},X_{s}^{x}=x,0\leq t\leq T,  \label{11SDE}
\end{equation}%
where $b$ is of the form (\ref{Decomp}) and $\rho _{1}\in \mathbb{R}%
\smallsetminus \left\{ 0\right\} ,\rho _{2}\in \mathbb{R}$. Let $K$ be a
compact cube in $\mathbb{R}$ and $r\in \mathbb{N}.$ Then for all $s\in \left[
0,T\right) $ and $x,y\in K$:%
\begin{equation}
E\left[ \sup_{t\in \left[ s,T\right] }\left\vert
X_{t}^{s,x}-X_{t}^{s,y}\right\vert ^{2^{r}}\right] \leq
C_{r,H,T}(K)\left\vert x-y\right\vert ^{2^{r}}\text{.}  \label{SEstimate}
\end{equation}
\end{proposition}

\bigskip

\begin{proof}
Let us assume without loss of generality that $s=0$, $T=1$, $\rho _{1},\rho
_{2}=1$. In proving this result we aim at employing the inequality of
Garsia-Rodemich-Rumsey (see Lemma \ref{GarciaRodemichRumsey} in the
Appendix) in the the case, when $d(t,s)=\left\vert t-s\right\vert ^{\frac{%
\varepsilon }{1+\varepsilon }}$, $0<\varepsilon <1$, $\Psi (x)=x^{^{\frac{%
4(1+\varepsilon )}{\varepsilon }}},x\geq 0$, $\Lambda =\left[ 0,1\right] $, $%
f(t)=\left\vert X_{t}^{x}-X_{t}^{y}\right\vert ,x,y\in K$, where $K$ is a
compact cube in $\mathbb{R}$. Then, $\sigma (r)\geq r^{\frac{1+\varepsilon }{%
\varepsilon }}$ and we obtain that%
\begin{equation*}
\left\vert f(t)-f(s)\right\vert \leq 18\int_{0}^{d(t,s)/2}\Psi ^{-1}\left( 
\frac{U}{(\sigma (r))^{2}}\right) dr,
\end{equation*}%
where%
\begin{equation*}
U=\int_{0}^{1}\int_{0}^{1}\Psi \left( \frac{\left\vert
f(t_{2})-f(t_{1})\right\vert }{d(t_{2},t_{1})}\right) dt_{2}dt_{1}.
\end{equation*}%
For $s=0$ we get that%
\begin{eqnarray*}
\left\vert f(t)\right\vert &\leq &18\int_{0}^{d(t,0)/2}\Psi ^{-1}\left( 
\frac{U}{(\sigma (r))^{2}}\right) dr+\left\vert f(0)\right\vert \\
&=&18\int_{0}^{d(t,0)/2}\Psi ^{-1}\left( \frac{U}{(\sigma (r))^{2}}\right)
dr+\left\vert x-y\right\vert .
\end{eqnarray*}%
Let $p=2^{r}>1$ with $r\in \mathbb{N}$ such that $p\frac{\varepsilon }{%
4(1+\varepsilon )}>1$. Then%
\begin{equation*}
\left\vert f(t)\right\vert ^{p}\leq C_{p}((\int_{0}^{d(t,0)/2}\Psi
^{-1}\left( \frac{U}{(\sigma (r))^{2}}\right) dr)^{p}+\left\vert
x-y\right\vert ^{p}).
\end{equation*}%
Thus,%
\begin{eqnarray*}
\sup_{0\leq t\leq 1}\left\vert f(t)\right\vert ^{p} &\leq
&C_{p}((\int_{0}^{1}\Psi ^{-1}\left( \frac{U}{(\sigma (r))^{2}}\right)
dr)^{p}+\left\vert x-y\right\vert ^{p}) \\
&=&C_{p}((\int_{0}^{1}\left( \frac{U}{(\sigma (r))^{2}}\right) ^{\frac{%
\varepsilon }{4(1+\varepsilon )}}dr)^{p}+\left\vert x-y\right\vert ^{p}) \\
&\leq &C_{p}((\int_{0}^{1}\left( \frac{1}{r^{^{\frac{2(1+\varepsilon )}{%
\varepsilon }}}}\right) ^{\frac{\varepsilon }{4(1+\varepsilon )}}dr)^{p}U^{p%
\frac{\varepsilon }{4(1+\varepsilon )}}+\left\vert x-y\right\vert ^{p}) \\
&=&C_{p}((\int_{0}^{1}\frac{1}{r^{^{\frac{1}{2}}}}dr)^{p}U^{p\frac{%
\varepsilon }{4(1+\varepsilon )}}+\left\vert x-y\right\vert ^{p}) \\
&\leq &C_{p}(U^{p\frac{\varepsilon }{4(1+\varepsilon )}}+\left\vert
x-y\right\vert ^{p}).
\end{eqnarray*}%
Further,%
\begin{eqnarray*}
U^{p\frac{\varepsilon }{4(1+\varepsilon )}} &\leq
&\int_{0}^{1}\int_{0}^{1}\left( \Psi \left( \frac{\left\vert
f(t_{2})-f(t_{1})\right\vert }{d(t_{2},t_{1})}\right) \right) ^{p\frac{%
\varepsilon }{4(1+\varepsilon )}}dt_{2}dt_{1} \\
&=&\int_{0}^{1}\int_{0}^{1}\left( \left( \frac{\left\vert
f(t_{2})-f(t_{1})\right\vert }{d(t_{2},t_{1})}\right) ^{\frac{%
4(1+\varepsilon )}{\varepsilon }}\right) ^{p\frac{\varepsilon }{%
4(1+\varepsilon )}}dt_{2}dt_{1} \\
&=&\int_{0}^{1}\int_{0}^{1}\left( \frac{\left\vert
f(t_{2})-f(t_{1})\right\vert }{d(t_{2},t_{1})}\right) ^{p}dt_{2}dt_{1}.
\end{eqnarray*}%
Hence, we find that%
\begin{eqnarray*}
E\left[ \sup_{0\leq t\leq 1}\left\vert f(t)\right\vert ^{p}\right] &\leq
&C_{p}(E\left[ U^{p\frac{\varepsilon }{2(1+\varepsilon )}}\right]
+\left\vert x-y\right\vert ^{p}) \\
&\leq &C_{p}(\int_{0}^{1}\int_{0}^{1}E\left[ \left( \frac{\left\vert
f(t_{2})-f(t_{1})\right\vert }{d(t_{2},t_{1})}\right) ^{p}\right]
dt_{2}dt_{1}+\left\vert x-y\right\vert ^{p}) \\
&\leq &C_{p}(\int_{0}^{1}\int_{0}^{1}E\left[ \left( \frac{\left\vert
X_{t_{2}}^{x}-X_{t_{2}}^{y}-(X_{t_{1}}^{x}-X_{t_{1}}^{y})\right\vert }{%
d(t_{2},t_{1})}\right) ^{p}\right] dt_{2}dt_{1}+\left\vert x-y\right\vert
^{p}).
\end{eqnarray*}%
Further, by applying \ref{Lto} in the proof of Theorem \ref{Ouknine}, the
uniform integrability of $\left\vert X_{t}^{x,n}\right\vert ^{p},n\geq
1,0\leq t\leq 1$ and Fatou`s Lemma, we see that%
\begin{eqnarray*}
&&E\left[ \sup_{0\leq t\leq 1}\left\vert f(t)\right\vert ^{p}\right] \\
&\leq &C_{p}(\underline{\lim }_{n\longrightarrow \infty
}\int_{0}^{1}\int_{0}^{1}E\left[ \left( \frac{\left\vert
X_{t_{2}}^{x,n}-X_{t_{2}}^{y,n}-(X_{t_{1}}^{x,n}-X_{t_{1}}^{y,n})\right\vert 
}{d(t_{2},t_{1})}\right) ^{p}\right] dt_{2}dt_{1}+\left\vert x-y\right\vert
^{p})\text{,}
\end{eqnarray*}%
where $X_{t}^{x,n},0\leq t\leq 1$ is the strong solution to (\ref{11SDE})
associated with the approximating sequence of smooth vector fields $%
b_{n},n\geq 1$ given by (\ref{b1}) and (\ref{b2}).

Assume without loss of generality that $x>y$. Then, by using the fundamental
theorem of calculus, we see that%
\begin{equation*}
X_{t_{2}}^{x,n}-X_{t_{2}}^{y,n}-(X_{t_{1}}^{x,n}-X_{t_{1}}^{y,n})=%
\int_{y}^{x}(\frac{\partial }{\partial z}X_{t_{2}}^{z,n}-\frac{\partial }{%
\partial z}X_{t_{1}}^{z,n})dz.
\end{equation*}%
Hence,%
\begin{eqnarray*}
&&E\left[ \left( \frac{\left\vert
X_{t_{2}}^{x,n}-X_{t_{2}}^{y,n}-(X_{t_{1}}^{x,n}-X_{t_{1}}^{y,n})\right\vert 
}{d(t_{2},t_{1})}\right) ^{p}\right] \\
&\leq &E\left[ \left( \int_{y}^{x}\frac{\left\vert \frac{\partial }{\partial
z}X_{t_{2}}^{z,n}-\frac{\partial }{\partial z}X_{t_{1}}^{z,n}\right\vert }{%
d(t_{2},t_{1})}dz\right) ^{p}\right] \\
&\leq &\left\vert x-y\right\vert ^{p-1}E\left[ \int_{y}^{x}\left( \frac{%
\left\vert \frac{\partial }{\partial z}X_{t_{2}}^{z,n}-\frac{\partial }{%
\partial z}X_{t_{1}}^{z,n}\right\vert }{d(t_{2},t_{1})}\right) ^{p}dz\right]
\\
&=&\left\vert x-y\right\vert ^{p-1}\int_{y}^{x}E\left[ \left( \frac{%
\left\vert \frac{\partial }{\partial z}X_{t_{2}}^{z,n}-\frac{\partial }{%
\partial z}X_{t_{1}}^{z,n}\right\vert }{d(t_{2},t_{1})}\right) ^{p}\right] dz
\\
&\leq &\left\vert x-y\right\vert ^{p}\sup_{z\in K}E\left[ \left( \frac{%
\left\vert \frac{\partial }{\partial z}X_{t_{2}}^{z,n}-\frac{\partial }{%
\partial z}X_{t_{1}}^{z,n}\right\vert }{d(t_{2},t_{1})}\right) ^{p}\right]
\end{eqnarray*}%
Thus%
\begin{eqnarray}
&&E\left[ \sup_{0\leq t\leq 1}\left\vert f(t)\right\vert ^{p}\right]  \notag
\\
&\leq &C_{p,d}(\underline{\lim }_{n\longrightarrow \infty
}\int_{0}^{1}\int_{0}^{1}\sup_{z\in K}E\left[ (\left\vert \frac{\partial }{%
\partial x}X_{t_{2}}^{z,n}-\frac{\partial }{\partial x}X_{t_{1}}^{z,n}\right%
\vert /d(t_{2},t_{1}))^{p})\right] dt_{2}dt_{1}\left\vert x-y\right\vert ^{p}
\notag \\
&&+\left\vert x-y\right\vert ^{p})\text{.}  \label{Garcia}
\end{eqnarray}%
By Lemma 2.6 in \cite{BMPD} that there exists a $\beta >0$ depending on $K$
such that%
\begin{equation*}
\sup_{x\in K}E\left[ \mathcal{E}\left(
\int_{0}^{T}b_{n}(u,x+B_{u}^{H}+W_{u})dW_{u}\right) ^{1+\beta }\right]
<\infty
\end{equation*}%
where $\mathcal{E}(M_{T})=\mathcal{E}_{T}(M)$ denotes the Doleans-Dade
exponential of a martingale $M_{\cdot }$. Moreover, the proof of Lemma 2.6
in \cite{BMPD} in connection with the properties (\ref{b1}), (\ref{b2}) show
that%
\begin{eqnarray}
&&E\left[ \mathcal{E}\left(
\int_{0}^{T}b_{n}(u,x+B_{u}^{H}+W_{u})dW_{u}\right) ^{1+\beta }\right] 
\notag \\
&\leq &e^{\widetilde{C}_{\beta ,T}T(1+\left\vert x\right\vert )^{2}}\times 
\notag \\
&&E\left[ \exp \left\{ 2\widetilde{C}_{\beta ,T}(1+\left\vert x\right\vert
)\int_{0}^{T}\left\vert B_{u}\right\vert du+\widetilde{C}_{\beta ,T}\right\}
\int_{0}^{T}\left\vert B_{u}\right\vert ^{2}du\right] ,  \label{DIneq}
\end{eqnarray}%
where $\widetilde{C}_{\beta ,T}$ is a constant with $\lim_{\beta \searrow 0}%
\widetilde{C}_{\beta ,T}=0$.

We know that%
\begin{equation*}
\frac{\partial }{\partial x}X_{t}^{x,n}=\exp (\int_{0}^{t}b_{n}^{\shortmid
}(u,X_{u}^{x,n})du),0\leq t\leq 1\text{.}
\end{equation*}%
Then, using Girsanov`s theorem with respect to the Brownian motion $W_{\cdot
}$ and H\"{o}lder`s inequality, we find that%
\begin{eqnarray}
&&E\left[ \left( \frac{\left\vert \frac{\partial }{\partial z}%
X_{t_{2}}^{x,n}-\frac{\partial }{\partial z}X_{t_{1}}^{x,n}\right\vert }{%
d(t_{2},t_{1})}\right) ^{p}\right]  \notag \\
&\leq &E\left[ \left( \frac{\left\vert \frac{\partial }{\partial z}%
X_{t_{2}}^{x,n}-\frac{\partial }{\partial z}X_{t_{1}}^{x,n}\right\vert }{%
d(t_{2},t_{1})}\right) ^{p\frac{1+\beta }{\beta }}\right] ^{\frac{\beta }{%
1+\beta }}E\left[ \mathcal{E}\left(
\int_{0}^{T}b_{n}(u,x+B_{u}^{H}+W_{u})dW_{u}\right) ^{1+\beta }\right] ^{%
\frac{1}{1+\beta }}  \notag \\
&=&E\left[ \left( \frac{\left\vert \exp (\int_{0}^{t_{2}}b_{n}^{\shortmid
}(u,x+B_{u}^{H}+W_{u})du)-\exp (\int_{0}^{t_{1}}b_{n}^{\shortmid
}(u,x+B_{u}^{H}+W_{u})du)\right\vert }{d(t_{2},t_{1})}\right) ^{p\frac{%
1+\beta }{\beta }}\right] ^{\frac{\beta }{1+\beta }}\times  \notag \\
&&E\left[ \mathcal{E}\left(
\int_{0}^{T}b_{n}(u,x+B_{u}^{H}+W_{u})dW_{u}\right) ^{1+\beta }\right] ^{%
\frac{1}{1+\beta }}\text{.}  \label{QR}
\end{eqnarray}%
Further, by applying the inequality $\left\vert
e^{z_{1}}-e^{z_{2}}\right\vert \leq (e^{z_{1}}\vee e^{z_{2}})\left\vert
z_{1}-z_{2}\right\vert $ and H\"{o}lder 
%TCIMACRO{\U{b4}}%
%BeginExpansion
\'{}%
%EndExpansion
s inequality, we have for $t_{2}>t_{1}$ that%
\begin{eqnarray*}
&&E\left[ \left( \frac{\left\vert \exp (\int_{0}^{t_{2}}b_{n}^{\shortmid
}(u,x+B_{u}^{H}+W_{u})du)-\exp (\int_{0}^{t_{1}}b_{n}^{\shortmid
}(u,x+B_{u}^{H}+W_{u})du)\right\vert }{d(t_{2},t_{1})}\right) ^{p\frac{%
1+\beta }{\beta }}\right] ^{\frac{\beta }{1+\beta }} \\
&\leq &E\left[ \left( \frac{\left\vert \int_{t_{1}}^{t_{2}}b_{n}^{\shortmid
}(u,x+B_{u}^{H}+W_{u})du\right\vert }{d(t_{2},t_{1})}\right) ^{2p\frac{%
1+\beta }{\beta }}\right] ^{\frac{\beta }{2(1+\beta )}}\times \\
&&E\left[ \left( \exp (\int_{0}^{t_{2}}b_{n}^{\shortmid
}(u,x+B_{u}^{H}+W_{u})du)\vee \exp (\int_{0}^{t_{1}}b_{n}^{\shortmid
}(u,x+B_{u}^{H}+W_{u})du)\right) ^{2p\frac{1+\beta }{\beta }}\right] ^{\frac{%
\beta }{2(1+\beta )}} \\
&=&I_{1}(x)\cdot I_{2}(x)\text{.}
\end{eqnarray*}%
Suppose without loss of generality that $p\frac{1+\beta }{\beta }=m$ for $%
m\in \mathbb{N}$. We observe that%
\begin{eqnarray*}
&&E\left[ \left( \frac{\left\vert \int_{t_{1}}^{t_{2}}b_{n}^{\shortmid
}(u,x+B_{u}^{H}+W_{u})du\right\vert }{d(t_{2},t_{1})}\right) ^{2p\frac{%
1+\beta }{\beta }}\right] \\
&\leq &C_{m}E\left[ \left( \frac{\left\vert \int_{t_{1}}^{t_{2}}\widetilde{b}%
_{n}^{\shortmid }(u,x+B_{u}^{H}+W_{u})du\right\vert }{d(t_{2},t_{1})}\right)
^{2p\frac{1+\beta }{\beta }}\right. \left. +\left( \frac{\left\vert
\int_{t_{1}}^{t_{2}}\widehat{b}_{n}^{\shortmid
}(u,x+B_{u}^{H}+W_{u})du\right\vert }{d(t_{2},t_{1})}\right) ^{2p\frac{%
1+\beta }{\beta }}\right] \\
&\leq &C_{m}(E\left[ \left( \frac{\left\vert \int_{t_{1}}^{t_{2}}\widetilde{b%
}_{n}^{\shortmid }(u,x+B_{u}^{H}+W_{u})du\right\vert }{d(t_{2},t_{1})}%
\right) ^{2p\frac{1+\beta }{\beta }}\right. +K^{2p\frac{1+\beta }{\beta }%
}\left( \frac{\left\vert t_{2}-t_{1}\right\vert }{d(t_{2},t_{1})}\right) ^{2p%
\frac{1+\beta }{\beta }}).
\end{eqnarray*}%
We now choose $\varepsilon >0$ such $\frac{\varepsilon }{1+\varepsilon }%
=1-3H $. So 
\begin{equation*}
d(t,s)=\left\vert t-s\right\vert ^{\frac{\varepsilon }{1+\varepsilon }%
}=\left\vert t-s\right\vert ^{1-3H}.
\end{equation*}%
On the other hand, it follows from inequality (\ref{Lemma 2.12}) in
connection with Lemma \ref{L. 2.12} in the Appendix and (\ref{b2}) that%
\begin{eqnarray*}
&&E\left[ \left( \frac{\left\vert \int_{t_{1}}^{t_{2}}\widetilde{b}%
_{n}^{\shortmid }(u,x+B_{u}^{H}+W_{u})du\right\vert }{d(t_{2},t_{1})}\right)
^{2p\frac{1+\beta }{\beta }}\right] \\
&\leq &L_{m}E\left[ \exp (\alpha C(H,d,T)\left\Vert \widetilde{b}%
_{n}\right\Vert _{\infty }^{2}(1+\sup_{0\leq l\leq T}\left\vert
B_{l}\right\vert )^{2})\right] \\
&\leq &L_{m}E\left[ \exp (\alpha C(H,d,T)L^{2}(1+\sup_{0\leq l\leq
T}\left\vert B_{l}\right\vert )^{2})\right] <\infty
\end{eqnarray*}%
for constants $L_{m},C(H,d,T),L,\alpha $, where $\alpha =\alpha (H,d,T)>0$
is sufficiently small and $d=1$. The latter implies that%
\begin{equation*}
\sup_{x\in K}I_{1}(x)\leq (L_{m}E\left[ \exp (\alpha
C(H,d,T)L^{2}(1+\sup_{0\leq l\leq T}\left\vert B_{l}\right\vert )^{2})\right]
)^{\frac{\beta }{1+\beta }}<\infty \text{.}
\end{equation*}%
As for the factor $I_{2}(x)$ we obtain that%
\begin{eqnarray*}
&&(I_{2}(x))^{2\frac{1+\beta }{\beta }} \\
&\leq &E\left[ \left( \exp (\int_{0}^{t_{2}}b_{n}^{\shortmid
}(u,x+B_{u}^{H}+W_{u})du)\vee \exp (\int_{0}^{t_{1}}b_{n}^{\shortmid
}(u,x+B_{u}^{H}+W_{u})du)\right) ^{2p\frac{1+\beta }{\beta }}\right] \\
&\leq &E\left[ \exp (2m\int_{0}^{t_{2}}b_{n}^{\shortmid
}(u,x+B_{u}^{H}+W_{u})du)\right] +E\left[ \exp
(2m\int_{0}^{t_{1}}b_{n}^{\shortmid }(u,x+B_{u}^{H}+W_{u})du)\right] \text{.}
\end{eqnarray*}%
We also see that%
\begin{eqnarray*}
&&E\left[ \exp (2m\int_{0}^{t}b_{n}^{\shortmid }(u,x+B_{u}^{H}+W_{u})du)%
\right] \\
&\leq &C_{m,K^{\ast },T}E\left[ \exp (2m\left\vert \int_{0}^{t}\widetilde{b}%
_{n}^{\shortmid }(u,x+B_{u}^{H}+W_{u})du\right\vert )\right] \text{,}
\end{eqnarray*}%
where $K^{\ast }$ is the uniform Lipschitz constant of $\widehat{b}_{n}$ in (%
\ref{b1}). Without loss of generality consider the case, when $t=1$. Then
inequality (\ref{Lemma 2.12}) in the Appendix combined with (\ref{b1}), (\ref%
{b2}) entail that 
\begin{eqnarray*}
&&E\left[ \exp (2m\left\vert \int_{0}^{1}\widetilde{b}_{n}^{\shortmid
}(u,x+B_{u}^{H}+W_{u})du\right\vert )\right] \\
&\leq &E\left[ \exp (\frac{2m}{\sqrt{\alpha }}+\alpha \left\vert \int_{0}^{1}%
\widetilde{b}_{n}^{\shortmid }(u,x+B_{u}^{H}+W_{u})du\right\vert ^{2})\right]
\\
&\leq &C_{\alpha ,m}E\left[ \exp (\alpha C(H,d,T)L^{2}(1+\sup_{0\leq l\leq
T}\left\vert B_{l}\right\vert )^{2})\right] ,
\end{eqnarray*}%
for constants $C_{\alpha ,m},C(H,d,T),L,\alpha $, where $\alpha =\alpha
(H,d,T)>0$ is sufficiently small and $d=1$. Hence, we find that%
\begin{equation*}
\sup_{x\in K}I_{2}(x)<\infty \text{.}
\end{equation*}%
Finally, the proof follows from (\ref{Garcia}), (\ref{QR}) and (\ref{DIneq}).
\end{proof}

\begin{remark}
An estimate of the form (\ref{SEstimate}) can be e.g. found in \cite{S1}, %
\ref{S2} in the case of a Wiener process and $b\in L^{\infty }(\left[ 0,T%
\right] \times \mathbb{R}^{d})$.\ See also \cite{AMP} in the case of a
fractional Browian motion with Hurst parameter $H<\frac{1}{2(d+2)}$ and $%
b\in L_{\infty ,\infty }^{1,\infty }=L_{\infty }^{1}\cap L_{\infty }^{\infty
}$. It turns out that such an estimate- as already mentioned in the
Introduction- plays a central role for proving \emph{path-by-path uniqueness}
of solutions (see \cite{Davie}) to SDE`s with additive Wiener or fractional
Brownian noise for bounded vector fields $b$, which is a much stronger
property than pathwise uniqueness. See \cite{S1}, \cite{S2} and also \cite%
{AMP} for more details.
\end{remark}

\section{Appendix}

\bigskip For some of the proofs in this article we need a version of
Girsanov's theorem for the fractional Brownian motion. In order to state
this result, let us recall some basic concepts from fractional calculus (see 
\cite{samko.et.al.93} and \cite{lizorkin.01}).

Let $a,$ $b\in \mathbb{R}$ with $a<b$. Let $f\in L^{p}([a,b])$ with $p\geq 1$
and $\alpha >0$. Define the \emph{left-} and \emph{right-sided
Riemann-Liouville fractional integrals} as 
\begin{equation*}
I_{a^{+}}^{\alpha }f(x)=\frac{1}{\Gamma (\alpha )}\int_{a}^{x}(x-y)^{\alpha
-1}f(y)dy
\end{equation*}%
and 
\begin{equation*}
I_{b^{-}}^{\alpha }f(x)=\frac{1}{\Gamma (\alpha )}\int_{x}^{b}(y-x)^{\alpha
-1}f(y)dy
\end{equation*}%
for almost all $x\in \lbrack a,b]$. Here $\Gamma $ denotes the Gamma
function.

Let $p\geq 1$ and let $I_{a^{+}}^{\alpha }(L^{p})$ (resp. $I_{b^{-}}^{\alpha
}(L^{p})$) be the image of $L^{p}([a,b])$ of the operator $I_{a^{+}}^{\alpha
}$ (resp. $I_{b^{-}}^{\alpha }$). If $f\in I_{a^{+}}^{\alpha }(L^{p})$
(resp. $f\in I_{b^{-}}^{\alpha }(L^{p})$) and $0<\alpha <1$ then we can
introduce the \emph{left-} and \emph{right-sided Riemann-Liouville
fractional derivatives} by 
\begin{equation*}
D_{a^{+}}^{\alpha }f(x)=\frac{1}{\Gamma (1-\alpha )}\frac{d}{dx}\int_{a}^{x}%
\frac{f(y)}{(x-y)^{\alpha }}dy
\end{equation*}%
and 
\begin{equation*}
D_{b^{-}}^{\alpha }f(x)=\frac{1}{\Gamma (1-\alpha )}\frac{d}{dx}\int_{x}^{b}%
\frac{f(y)}{(y-x)^{\alpha }}dy.
\end{equation*}

The left- and right-sided derivatives of $f$ have the representations 
\begin{equation*}
D_{a^{+}}^{\alpha }f(x)=\frac{1}{\Gamma (1-\alpha )}\left( \frac{f(x)}{%
(x-a)^{\alpha }}+\alpha \int_{a}^{x}\frac{f(x)-f(y)}{(x-y)^{\alpha +1}}%
dy\right)
\end{equation*}%
and 
\begin{equation*}
D_{b^{-}}^{\alpha }f(x)=\frac{1}{\Gamma (1-\alpha )}\left( \frac{f(x)}{%
(b-x)^{\alpha }}+\alpha \int_{x}^{b}\frac{f(x)-f(y)}{(y-x)^{\alpha +1}}%
dy\right) .
\end{equation*}

The above definitions entail that 
\begin{equation*}
I_{a^{+}}^{\alpha }(D_{a^{+}}^{\alpha }f)=f
\end{equation*}%
for all $f\in I_{a^{+}}^{\alpha }(L^{p})$ and 
\begin{equation*}
D_{a^{+}}^{\alpha }(I_{a^{+}}^{\alpha }f)=f
\end{equation*}%
for all $f\in L^{p}([a,b])$ and similarly for $I_{b^{-}}^{\alpha }$ and $%
D_{b^{-}}^{\alpha }$.

\bigskip

Denote by $B^{H}=\{B_{t}^{H},t\in \lbrack 0,T]\}$ a $d$-dimensional \emph{%
fractional Brownian motion} with Hurst parameter $H\in (0,1/2)$. Here this
means that $B_{\cdot }^{H}$ is a centered Gaussian process with a covariance
function given by 
\begin{equation*}
(R_{H}(t,s))_{i,j}:=E[B_{t}^{H,(i)}B_{s}^{H,(j)}]=\delta _{ij}\frac{1}{2}%
\left( t^{2H}+s^{2H}-|t-s|^{2H}\right) ,\quad i,j=1,\dots ,d,
\end{equation*}%
where $\delta _{ij}$ is one, if $i=j$, or zero else.

In what follows we briefly pass in review the construction of the fractional
Brownian motion, which can be found in \cite{Nualart}. For convenience, we
confine ourselves to the case $d=1$.

Let $\mathcal{E}$ be the class of step functions on $[0,T]$ and denote by $%
\mathcal{H}$ the Hilbert space which obtained through the completion of $%
\mathcal{E}$ with respect to the inner product 
\begin{equation*}
\langle 1_{[0,t]},1_{[0,s]}\rangle _{\mathcal{H}}=R_{H}(t,s).
\end{equation*}%
The latter gives an extension of the mapping $1_{[0,t]}\mapsto B_{t}$ to an
isometry between $\mathcal{H}$ and a Gaussian subspace of $L^{2}(\Omega )$
with respect to $B^{H}$. Let $\varphi \mapsto B^{H}(\varphi )$ be this
isometry.

If $H<1/2$, one verifies that the covariance function $R_{H}(t,s)$ can be
represented as

\bigskip\ 
\begin{equation}
R_{H}(t,s)=\int_{0}^{t\wedge s}K_{H}(t,u)K_{H}(s,u)du,  \label{2.2}
\end{equation}%
where 
\begin{equation}
K_{H}(t,s)=c_{H}\left[ \left( \frac{t}{s}\right) ^{H-\frac{1}{2}}(t-s)^{H-%
\frac{1}{2}}+\left( \frac{1}{2}-H\right) s^{\frac{1}{2}-H}\int_{s}^{t}u^{H-%
\frac{3}{2}}(u-s)^{H-\frac{1}{2}}du\right] .  \label{KH}
\end{equation}%
Here $c_{H}=\sqrt{\frac{2H}{(1-2H)\beta (1-2H,H+1/2)}}$ and $\beta $ is the
Beta function. See \cite[Proposition 5.1.3]{Nualart}.

Using the kernel $K_{H}$, one can define via (\ref{2.2}) an isometry $%
K_{H}^{\ast }$ between $\mathcal{E}$ and $L^{2}([0,T])$ such that $%
(K_{H}^{\ast }1_{[0,t]})(s)=K_{H}(t,s)1_{[0,t]}(s).$ This isometry admits
for an extension to the Hilbert space $\mathcal{H}$, which has the following
representations in terms of fractional derivatives

\begin{equation*}
(K_{H}^{\ast }\varphi )(s)=c_{H}\Gamma \left( H+\frac{1}{2}\right) s^{\frac{1%
}{2}-H}\left( D_{T^{-}}^{\frac{1}{2}-H}u^{H-\frac{1}{2}}\varphi (u)\right)
(s)
\end{equation*}%
and 
\begin{align*}
(K_{H}^{\ast }\varphi )(s)=& \,c_{H}\Gamma \left( H+\frac{1}{2}\right)
\left( D_{T^{-}}^{\frac{1}{2}-H}\varphi (s)\right) (s) \\
& +c_{H}\left( \frac{1}{2}-H\right) \int_{s}^{T}\varphi (t)(t-s)^{H-\frac{3}{%
2}}\left( 1-\left( \frac{t}{s}\right) ^{H-\frac{1}{2}}\right) dt.
\end{align*}%
for $\varphi \in \mathcal{H}$. One can also show that $\mathcal{H}%
=I_{T^{-}}^{\frac{1}{2}-H}(L^{2})$. See \cite{DU} and \cite[Proposition 6]%
{alos.mazet.nualart.01}.

We know that $K_{H}^{\ast }$ is an isometry from $\mathcal{H}$ into $%
L^{2}([0,T])$. Hence, the $d$-dimensional process $W=\{W_{t},t\in \lbrack
0,T]\}$ defined by 
\begin{equation}
W_{t}:=B^{H}((K_{H}^{\ast })^{-1}(1_{[0,t]}))  \label{WBH}
\end{equation}%
is a Wiener process and the process $B^{H}$ has the representation 
\begin{equation}
B_{t}^{H}=\int_{0}^{t}K_{H}(t,s)dW_{s}.  \label{BHW}
\end{equation}%
See \cite{alos.mazet.nualart.01}.

In the sequel, we also need the Definition of a fractional Brownian motion
with respect to a filtration.

\begin{definition}
Let $\mathcal{G}=\left\{ \mathcal{G}_{t}\right\} _{t\in \left[ 0,T\right] }$
be a filtration on $\left( \Omega ,\mathcal{F},P\right) $ satisfying the
usual conditions. A fractional Brownian motion $B^{H}$ is called a $\mathcal{%
G}$-fractional Brownian motion if the process $W$ defined by (\ref{WBH}) is
a $\mathcal{G}$-Brownian motion.
\end{definition}

\bigskip

Let $W$ be a standard Wiener process on a filtered probability space $%
(\Omega ,\mathfrak{A},P),\{\mathcal{F}_{t}\}_{t\in \lbrack 0,T]},$ where $%
\mathcal{F}=\{\mathcal{F}_{t}\}_{t\in \lbrack 0,T]}$ is the natural
filtration generated by $W$ and augmented by all $P$-null sets. Denote by $%
B:=B^{H}$ the fractional Brownian motion with Hurst parameter $H\in (0,1/2)$
as in (\ref{BHW}).

We shall also apply a version of Girsanov's theorem for fractional Brownian
motion which can be found in \cite[Theorem 4.9]{DU}. See also \cite[Theorem 2%
]{NO}. This version relies on the folloing definition of an isomorphism $%
K_{H}$ from $L^{2}([0,T])$ onto $I_{0+}^{H+\frac{1}{2}}(L^{2})$ with respect
to the kernel $K_{H}(t,s)$ based on fractional integrals (see \cite[Theorem
2.1]{DU}): 
\begin{equation*}
(K_{H}\varphi )(s)=I_{0^{+}}^{2H}s^{\frac{1}{2}-H}I_{0^{+}}^{\frac{1}{2}%
-H}s^{H-\frac{1}{2}}\varphi ,\quad \varphi \in L^{2}([0,T]).
\end{equation*}

The latter combined with the properties of the Riemann-Liouville fractional
integrals and derivatives can be used to prove the following representation
of inverse of $K_{H}$ : 
\begin{equation}
(K_{H}^{-1}\varphi )(s)=s^{\frac{1}{2}-H}D_{0^{+}}^{\frac{1}{2}-H}s^{H-\frac{%
1}{2}}D_{0^{+}}^{2H}\varphi (s),\quad \varphi \in I_{0+}^{H+\frac{1}{2}%
}(L^{2}).  \label{KHminus}
\end{equation}

Using this we find for absolutely continuous functions $\varphi $ (see \cite%
{NO}) that 
\begin{equation*}
(K_{H}^{-1}\varphi )(s)=s^{H-\frac{1}{2}}I_{0^{+}}^{\frac{1}{2}-H}s^{\frac{1%
}{2}-H}\varphi ^{\prime }(s).
\end{equation*}

\begin{theorem}[Girsanov's theorem for fBm]
\label{girsanov} Let $u=\{u_{t},t\in \lbrack 0,T]\}$ be an $\mathcal{F}$%
-adapted process with integrable trajectories and set $\widetilde{B}%
_{t}^{H}=B_{t}^{H}+\int_{0}^{t}u_{s}ds,\quad t\in \lbrack 0,T].$ Suppose that

\begin{itemize}
\item[(i)] $\int_{0}^{\cdot }u_{s}ds\in I_{0+}^{H+\frac{1}{2}}(L^{2}([0,T]))$%
, $P$-a.s.

\item[(ii)] $E[\xi_T]=1$ where 
\begin{equation*}
\xi_T := \exp\left\{-\int_0^T K_H^{-1}\left( \int_0^{\cdot} u_r
dr\right)(s)dW_s - \frac{1}{2} \int_0^T K_H^{-1} \left( \int_0^{\cdot} u_r
dr \right)^2(s)ds \right\}.
\end{equation*}
\end{itemize}

Then the shifted process $\widetilde{B}^H$ is an $\mathcal{F}$-fractional
Brownian motion with Hurst parameter $H$ under the new probability $%
\widetilde{P}$ defined by $\frac{d\widetilde{P}}{dP}=\xi_T$.
\end{theorem}

\begin{remark}
In the the multi-dimensional case, we define 
\begin{equation*}
(K_{H}\varphi )(s):=((K_{H}\varphi ^{(1)})(s),\dots ,(K_{H}\varphi
^{(d)})(s))^{\ast },\quad \varphi \in L^{2}([0,T];\mathbb{R}^{d}),
\end{equation*}%
where $\ast $ denotes transposition. Similarly for $K_{H}^{-1}$ and $%
K_{H}^{\ast }$.
\end{remark}

In this Appendix we also recapitulate an integration by parts formula from 
\cite{BNP}, which is based on a sort of local time of the Gaussian process $%
B^{H}$.

Let $m\in \mathbb{N}$ and let $f:[0,T]^{m}\times (\mathbb{R}%
^{d})^{m}\rightarrow \mathbb{R}$ be a function defined by 
\begin{equation*}
f(s,z)=\prod_{j=1}^{m}f_{j}(s_{j},z_{j}),\quad s=(s_{1},\dots ,s_{m})\in
\lbrack 0,T]^{m},\quad z=(z_{1},\dots ,z_{m})\in (\mathbb{R}^{d})^{m},
\end{equation*}%
where $f_{j}:[0,T]\times \mathbb{R}^{d}\rightarrow \mathbb{R}$, $j=1,\dots
,m $ are (spatially) smooth functions with compact support. Further, let $%
\varkappa :[0,T]^{m}\rightarrow \mathbb{R}$ be a function given by 
\begin{equation*}
\varkappa (s)=\prod_{j=1}^{m}\varkappa _{j}(s_{j}),\quad s\in \lbrack
0,T]^{m},
\end{equation*}%
where $\varkappa _{j}:[0,T]\rightarrow \mathbb{R}$, $j=1,\dots ,m$ are
integrable.

Further, denote by $\alpha _{j}$ a multiindex and $D^{\alpha _{j}}$ its
corresponding differential operator. Let $\alpha =(\alpha _{1},\dots ,\alpha
_{m})\in \mathbb{N}_{0}^{d\times m}$. Then define $|\alpha
|=\sum_{j=1}^{m}\sum_{l=1}^{d}\alpha _{j}^{(l)}$ and 
\begin{equation*}
D^{\alpha }f(s,z)=\prod_{j=1}^{m}D^{\alpha _{j}}f_{j}(s_{j},z_{j}).
\end{equation*}

In \cite{BNP} the following integration by parts formula was shown%
\begin{equation}
\int_{\Delta _{\theta ,t}^{m}}D^{\alpha }f(s,B_{s}^{H})ds=\int_{\left( 
\mathbb{R}^{d}\right) ^{m}}\Lambda _{\alpha }^{f}(\theta ,t,z)dz,
\label{ibp}
\end{equation}%
where $\Lambda _{\alpha }^{f}$ is a suitable random field, $\Delta _{\theta
,t}^{m}$ is the $m$-dimensional simplex (see \ref{Simplex}) and $%
B_{s}^{H}:=(B_{s_{1}}^{H},...,B_{s_{m}}^{H})$ is a fractional Brownian on
that simplex. More specifically, we have that 
\begin{equation}
\Lambda _{\alpha }^{f}(\theta ,t,z)=(2\pi )^{-dm}\int_{(\mathbb{R}%
^{d})^{m}}\int_{\Delta _{\theta
,t}^{m}}\prod_{j=1}^{m}f_{j}(s_{j},z_{j})(-iu_{j})^{\alpha _{j}}\exp
\{-i\langle u_{j},B_{s_{j}}^{H}-z_{j}\rangle \}dsdu.  \label{LambdaDef}
\end{equation}

It turns out that the ramdom field $\Lambda _{\alpha }^{f}(\theta ,t,z)$ is
a well-defined element of $L^{2}(\Omega )$.

In this paper, we also need the notion of \emph{shuffle permutations}: Let $%
m $ and $n$ be integers. Denote by $S(m,n)$ the set of shuffle permutations,
that is the set of permutations $\sigma :\{1,\dots ,m+n\}\rightarrow
\{1,\dots ,m+n\}$ such that $\sigma (1)<\dots <\sigma (m)$ and $\sigma
(m+1)<\dots <\sigma (m+n)$.

We introduce the following notation: Given $(s,z)=(s_{1},\dots
,s_{m},z_{1}\dots ,z_{m})\in \lbrack 0,T]^{m}\times (\mathbb{R}^{d})^{m}$
and a shuffle $\sigma \in S(m,m)$ we write 
\begin{equation*}
f_{\sigma }(s,z):=\prod_{j=1}^{2m}f_{[\sigma (j)]}(s_{j},z_{[\sigma (j)]})
\end{equation*}%
and 
\begin{equation*}
\varkappa _{\sigma }(s):=\prod_{j=1}^{2m}\varkappa _{\lbrack \sigma
(j)]}(s_{j}),
\end{equation*}%
where $[j]$ is equal to $j$ if $1\leq j\leq m$ and $j-m$ if $m+1\leq j\leq
2m $.

Define the expressions 
\begin{eqnarray*}
&&\Psi _{k}^{f}(\theta ,t,z) \\
&:&=\prod_{l=1}^{d}\sqrt{(2\left\vert \alpha ^{(l)}\right\vert )!}%
\sum_{\sigma \in S(m,m)}\int_{\Delta _{0,t}^{2m}}\left\vert f_{\sigma
}(s,z)\right\vert \prod_{j=1}^{2m}\frac{1}{\left\vert
s_{j}-s_{j-1}\right\vert ^{H(d+2\sum_{l=1}^{d}\alpha _{\lbrack \sigma
(j)]}^{(l)})}}ds_{1}...ds_{2m},
\end{eqnarray*}

\begin{eqnarray*}
&&\Psi _{k}^{\varkappa }(\theta ,t) \\
&:&=\prod_{l=1}^{d}\sqrt{(2\left\vert \alpha ^{(l)}\right\vert )!}%
\sum_{\sigma \in S(m,m)}\int_{\Delta _{0,t}^{2m}}\left\vert \varkappa
_{\sigma }(s)\right\vert \prod_{j=1}^{2m}\frac{1}{\left\vert
s_{j}-s_{j-1}\right\vert ^{H(d+2\sum_{l=1}^{d}\alpha _{\lbrack \sigma
(j)]}^{(l)})}}ds_{1}...ds_{2m}.
\end{eqnarray*}

\bigskip

\begin{theorem}
\bigskip \label{mainestimate}Suppose that $\Psi _{k}^{f}(\theta ,t,z),\Psi
_{k}^{\varkappa }(\theta ,t)<\infty $. Then, $\Lambda _{\alpha }^{f}(\theta
,t,z)$ given by (\ref{LambdaDef}) is a random variable in $L^{2}(\Omega )$
and there exists a universal constant $C=C(T;H,d)>0$ such that%
\begin{equation*}
E\left[ \left\vert \Lambda _{\alpha }^{f}(\theta ,t,z)\right\vert ^{2}\right]
\leq C^{m+\left\vert \alpha \right\vert }\Psi _{k}^{f}(\theta ,t,z)\text{.}
\end{equation*}%
Moreover, we have%
\begin{eqnarray*}
&&\left\vert E\left[ \int_{(\mathbb{R}^{d})^{m}}\Lambda _{\alpha
}^{f}(\theta ,t,z)dz\right] \right\vert \\
&\leq &C^{m/2+\left\vert \alpha \right\vert
/2}\dprod\limits_{j=1}^{m}\left\Vert f_{j}\right\Vert _{L^{1}(\mathbb{R}%
^{d};L^{\infty }(\left[ 0,T\right] ))}(\Psi _{k}^{\varkappa }(\theta
,t))^{1/2}.
\end{eqnarray*}
\end{theorem}

\bigskip

We also need the following auxiliary result in connection with proof of
Theorem \ref{Bismut}

\begin{lemma}
\label{WeakConvergence}\bigskip\ Let $U\subset \mathbb{R}^{d}$ be an open
and bounded subset. Consider the sequence $X_{\cdot }^{x,n},n\geq 1$ in
Proposition \ref{ConvergenceX}. Then%
\begin{equation*}
\frac{\partial }{\partial x}X_{\cdot }^{\cdot ,n}\underset{n\longrightarrow
\infty }{\longrightarrow }\frac{\partial }{\partial x}X_{\cdot }^{\cdot }
\end{equation*}%
in $L^{2}([0,T]\times \Omega \times U)$ weakly.

\begin{proof}
This result is a consequence of relation (\ref{ConvergenceX}) and the
estimate (\ref{DerivativeEstimate}) in step 3 of the proof sketch of Theorem %
\ref{StrongSolution}.
\end{proof}
\end{lemma}

\bigskip

\bigskip

\begin{proof}[Proof of Theorem \protect\ref{Bismut}]
Suppose that $\Phi \in C_{c}^{\infty }(\mathbb{R}^{d})$ and let $b_{n}\in
C_{c}^{\infty }((0,T)\times \mathbb{R}^{d})$ be a sequence of functions such
that 
\begin{equation*}
b_{n}\underset{n\longrightarrow \infty }{\longrightarrow }b\text{ in }%
L_{\infty }^{1}\text{.}
\end{equation*}%
Let $X_{\cdot }^{s,x,n}$ be the unique strong solution to 
\begin{equation*}
dX_{t}^{s,x,n}=b_{n}(t,X_{t}^{s,x,n})dt+dB_{t}^{H},X_{s}^{s,x,n}=x,s\leq
t\leq T
\end{equation*}%
for all $n$. Since $b_{n}$ is smooth and compactly supported,\ we know (see
e.g.\cite{Kunita}) that there exists a $\Omega ^{\ast }$ with $\mu (\Omega
^{\ast })=1$ such that for all $\omega \in \Omega ^{\ast },0\leq s\leq t\leq
T$%
\begin{equation*}
(x\mapsto X_{t}^{s,x,n}(\omega ))\in C^{\infty }(\mathbb{R}^{d}).
\end{equation*}%
Hence dominated convergence implies that%
\begin{equation*}
\frac{\partial }{\partial x}E[\Phi (X_{T}^{x,n})]=E[\Phi ^{\shortmid
}(X_{T}^{x,n})\frac{\partial }{\partial x}X_{T}^{x,n}]\text{.}
\end{equation*}%
Here $\Phi ^{\shortmid }$ denotes the derivative of $\Phi $ and $%
X_{t}^{x,n}=X_{t}^{0,x,n}$. Further, we obtain for all $0\leq s\leq t\leq
T,x\in U$ that%
\begin{equation*}
X_{t}^{x,n}=X_{t}^{s,X_{s}^{x,n},n}\text{ a.e.}
\end{equation*}%
Thus%
\begin{equation*}
\frac{\partial }{\partial x}E[\Phi (X_{T}^{x,n})]=E[\Phi ^{\shortmid
}(X_{T}^{x,n})\frac{\partial }{\partial x}X_{T}^{s,X_{s}^{x,n},n}\frac{%
\partial }{\partial x}X_{s}^{x,n}].
\end{equation*}%
It is known that the Malliavin derivative $D_{\cdot }^{H}X_{t}^{s,x,n}$ of $%
X_{t}^{s,x,n}$ in the direction of $B_{\cdot }^{H}$ exists (see e.g. \cite{N}%
) and that%
\begin{equation*}
D_{u}^{H}X_{t}^{s,x,n}=\int_{u}^{t}b_{n}^{\shortmid
}(t,X_{r}^{s,x,n})D_{u}^{H}X_{r}^{s,x,n}dr+\chi _{_{(s,t]}(u)}I_{d\times d},
\end{equation*}%
where $I_{d\times d}$ is the identity matrix. We also observe that $\frac{%
\partial }{\partial x}X_{\cdot }^{u,X_{u}^{x,n},n}$ is a solution of the
latter equation for $s=0$. So we conclude that%
\begin{equation*}
D_{u}^{H}X_{t}^{x,n}=\frac{\partial }{\partial x}X_{t}^{u,X_{u}^{x,n},n}
\end{equation*}%
a.e., which yields%
\begin{equation*}
\frac{\partial }{\partial x}E[\Phi (X_{T}^{x,n})]=E[\Phi ^{\shortmid
}(X_{T}^{x,n})D_{s}^{H}X_{T}^{x,n}\frac{\partial }{\partial x}X_{s}^{x,n}].
\end{equation*}%
Choose $\varphi \in C_{c}^{\infty }(U)$. Then, we see that%
\begin{equation*}
-\int_{U}E[\Phi (X_{T}^{x,n})]\frac{\partial }{\partial x}\varphi
(x)dx=\int_{U}\varphi (x)E[\Phi ^{\shortmid
}(X_{T}^{x,n})D_{s}^{H}X_{T}^{x,n}\frac{\partial }{\partial x}X_{s}^{x,n}]dx.
\end{equation*}%
Since the function $a$ sums up to one, we can then use the chain rule for $%
D_{\cdot }^{H}$ (see \cite{Nualart}) and get that%
\begin{eqnarray*}
&&-\int_{U}E[\Phi (X_{T}^{x,n})]\frac{\partial }{\partial x}\varphi (x)dx \\
&=&\int_{U}\varphi (x)E[\int_{0}^{T}\{a(s)\Phi ^{\shortmid
}(X_{T}^{x,n})D_{s}^{H}X_{T}^{x,n}\frac{\partial }{\partial x}%
X_{s}^{x,n}\}ds]dx \\
&=&\int_{U}\varphi (x)E[\int_{0}^{T}\{a(s)D_{s}^{H}\Phi (X_{T}^{x,n})\frac{%
\partial }{\partial x}X_{s}^{x,n}\}ds]dx
\end{eqnarray*}

Further, we know from Proposition 5.2.1 and p. 285 in \cite{Nualart} that%
\begin{equation*}
D_{s}^{H}\Phi (X_{T}^{x,n})=Cs^{\frac{1}{2}-H}(\int_{s}^{T}(u-s)^{-H-\frac{1%
}{2}}u^{H-\frac{1}{2}}D_{u}\Phi (X_{T}^{x,n})du.
\end{equation*}%
for a constant $C$ depending on $H$. Therefore, using substitution (first
for $u$ substituted by $u+s$ in the above relation and then for $s$ by $s-u$
in the next step), Fubini's theorem and the duality formula with respect to
the Malliavin derivative $D_{\cdot }$ we obtain that%
\begin{eqnarray*}
&&-\int_{U}E[\Phi (X_{T}^{x,n})]\frac{\partial }{\partial x}\varphi (x)dx \\
&=&C\int_{U}\varphi (x)E[\int_{0}^{T}\{a(s)Cs^{\frac{1}{2}-H} \\
&&\times (\int_{s}^{T}(u-s)^{-H-\frac{1}{2}}u^{H-\frac{1}{2}}D_{u}\Phi
(X_{T}^{x,n})du)\frac{\partial }{\partial x}X_{s}^{x,n}\}ds]dx \\
&=&C\int_{U}\varphi (x)E[\int_{0}^{T}u^{-H-\frac{1}{2}} \\
&&\times \int_{u}^{T}a(s-u)(s-u)^{\frac{1}{2}-H}s^{H-\frac{1}{2}}D_{s}\Phi
(X_{T}^{x,n})\frac{\partial }{\partial x}X_{s-u}^{x,n}dsdu]dx \\
&=&C\int_{U}\varphi (x)E[\Phi (X_{T}^{x,n}) \\
&&\times \int_{0}^{T}u^{-H-\frac{1}{2}}\int_{u}^{T}a(s-u)(s-u)^{\frac{1}{2}%
-H}s^{H-\frac{1}{2}}\left( \frac{\partial }{\partial x}X_{s-u}^{x,n}\right)
^{\ast }dB_{s}du]^{\ast }dx \\
&=&I_{1}(n)+I_{2}(n),
\end{eqnarray*}%
where%
\begin{eqnarray*}
I_{1}(n) &:&=C\int_{U}\varphi (x)E[(\Phi (X_{T}^{x,n})-\Phi (X_{T}^{x})) \\
&&\times \int_{0}^{T}u^{-H-\frac{1}{2}}\int_{u}^{T}a(s-u)(s-u)^{\frac{1}{2}%
-H}s^{H-\frac{1}{2}}\left( \frac{\partial }{\partial x}X_{s-u}^{x,n}\right)
^{\ast }dB_{s}du]^{\ast }dx
\end{eqnarray*}%
and%
\begin{eqnarray*}
I_{2}(n) &:&=C\int_{U}\varphi (x)E[\Phi (X_{T}^{x})\int_{0}^{T}u^{-H-\frac{1%
}{2}} \\
&&\times \int_{u}^{T}a(s-u)(s-u)^{\frac{1}{2}-H}s^{H-\frac{1}{2}}\left( 
\frac{\partial }{\partial x}X_{s-u}^{x,n}\right) ^{\ast }dB_{s}du]^{\ast }dx
\\
&=&C\int_{U}\varphi (x)E[\Phi (X_{T}^{x})\int_{0}^{T}u^{-H-\frac{1}{2}} \\
&&\times \int_{u}^{T}a(s-u)(s-u)^{\frac{1}{2}-H}s^{H-\frac{1}{2}}\left( 
\frac{\partial }{\partial x}X_{s-u}^{x}\right) ^{\ast }dB_{s}du]^{\ast }dx \\
&&+I_{3}(n),
\end{eqnarray*}%
where%
\begin{eqnarray*}
&&I_{3}(n) \\
&=&C\int_{U}\varphi (x)E[\Phi (X_{T}^{x})\int_{0}^{T}u^{-H-\frac{1}{2}%
}\int_{u}^{T}a(s-u)(s-u)^{\frac{1}{2}-H}s^{H-\frac{1}{2}} \\
&&\times \{\left( \frac{\partial }{\partial x}X_{s-u}^{x,n}\right) ^{\ast
}-\left( \frac{\partial }{\partial x}X_{s-u}^{x}\right) ^{\ast
}\}dB_{s}du]^{\ast }dx.
\end{eqnarray*}%
By applying Fubini's theorem, H\"{o}lder's inequality, the It\^{o} isometry,
the estimate (\ref{DerivativeEstimate}), the relation (\ref{ConvergenceX})
in step 3 of the proof sketch of Theorem \ref{StrongSolution} and dominated
convergence that%
\begin{eqnarray*}
&&\left\Vert I_{1}(n)\right\Vert \\
&\leq &\left\Vert \varphi \right\Vert _{\infty }\int_{U}(E[\left\vert \Phi
(X_{T}^{x,n})-\Phi (X_{T}^{x})\right\vert ^{2}])^{1/2} \\
&&\times (\int_{0}^{T}s^{2H-1}E[(\int_{0}^{s}u^{-H-\frac{1}{2}}\left\vert
a(s-u)\right\vert (s-u)^{\frac{1}{2}-H}\left\Vert \frac{\partial }{\partial x%
}X_{s-u}^{x,n}\right\Vert du)^{2}]ds)^{1/2}dx \\
&\leq &\left\Vert \varphi \right\Vert _{\infty }\int_{U}(E[\left\vert \Phi
(X_{T}^{x,n})-\Phi (X_{T}^{x})\right\vert ^{2}])^{1/2}(\int_{0}^{T}s^{2H-1}
\\
&&\times \int_{0}^{s}\int_{0}^{s}u_{1}^{-H-\frac{1}{2}}\left\vert
a(s-u_{1})\right\vert (s-u_{1})^{\frac{1}{2}-H}u_{2}^{-H-\frac{1}{2}%
}\left\vert a(s-u_{2})\right\vert (s-u_{2})^{\frac{1}{2}-H} \\
&&\times E[\left\Vert \frac{\partial }{\partial x}X_{s-u_{1}}^{x,n}\right%
\Vert ^{2}]^{1/2}E[\left\Vert \frac{\partial }{\partial x}%
X_{s-u_{2}}^{x,n}\right\Vert ^{2}]^{1/2}du_{1}du_{2}ds)^{1/2}dx \\
&=&\left\Vert \varphi \right\Vert _{\infty }\int_{U}(E[\left\vert \Phi
(X_{T}^{x,n})-\Phi (X_{T}^{x})\right\vert ^{2}])^{1/2}(\int_{0}^{T}s^{2H-1}
\\
&&\times (\int_{0}^{s}u^{-H-\frac{1}{2}}\left\vert a(s-u)\right\vert (s-u)^{%
\frac{1}{2}-H}E[\left\Vert \frac{\partial }{\partial x}X_{s-u}^{x,n}\right%
\Vert ^{2}]^{1/2}du)^{2}ds)^{1/2}dx \\
&\leq &\left\Vert \varphi \right\Vert _{\infty }\int_{U}(E[\left\vert \Phi
(X_{T}^{x,n})-\Phi (X_{T}^{x})\right\vert ^{2}])^{1/2}dx(\int_{0}^{T}s^{2H-1}
\\
&&\times \sup_{n\geq 1}L_{2,H,d,T}(\left\Vert b_{n}\right\Vert _{L_{\infty
}^{1}})^{1/4}(\int_{0}^{s}u^{-H-\frac{1}{2}}\left\vert a(s-u)\right\vert
(s-u)^{\frac{1}{2}-H}du)^{2}ds)^{1/2} \\
&\leq &C\left\Vert \varphi \right\Vert _{\infty }\int_{U}(E[\left\vert \Phi
(X_{T}^{x,n})-\Phi (X_{T}^{x})\right\vert ^{2}])^{1/2}dx(\int_{0}^{T}s^{H-%
\frac{1}{2}}ds)^{1/2} \\
&&\underset{n\longrightarrow \infty }{\longrightarrow }0,
\end{eqnarray*}%
where used the boundedness of the function $a$ in the last estimate.

Using the Clark-Ocone formula (see e.g. \cite{Nualart}) combined with the It%
\^{o} isometry and the chain rule for Malliavin derivatives, we find that%
\begin{eqnarray*}
&&I_{3}(n) \\
&:&=C\int_{U}\varphi (x)E[E[\Phi (X_{T}^{x})]\int_{0}^{T}u^{-H-\frac{1}{2}%
}\int_{u}^{T}a(s-u)(s-u)^{\frac{1}{2}-H}s^{H-\frac{1}{2}} \\
&&\times \{\left( \frac{\partial }{\partial x}X_{s-u}^{x,n}\right) ^{\ast
}-\left( \frac{\partial }{\partial x}X_{s-u}^{x}\right) ^{\ast
}\}dB_{s}du]^{\ast }dx \\
&&+C\int_{U}\varphi (x)E[\int_{0}^{T}u^{-H-\frac{1}{2}%
}\int_{u}^{T}a(s-u)(s-u)^{\frac{1}{2}-H}s^{H-\frac{1}{2}}D_{s}\Phi
(X_{T}^{x}) \\
&&\times \{\frac{\partial }{\partial x}X_{s-u}^{x,n}-\frac{\partial }{%
\partial x}X_{s-u}^{x}\}^{\ast }dsdu]^{\ast }dx \\
&=&C\int_{U}\varphi (x)E[\int_{0}^{T}u^{-H-\frac{1}{2}%
}\int_{u}^{T}a(s-u)(s-u)^{\frac{1}{2}-H}s^{H-\frac{1}{2}}\Phi ^{\shortmid
}(X_{T}^{x})D_{s}X_{T}^{x} \\
&&\times \{\frac{\partial }{\partial x}X_{s-u}^{x,n}-\frac{\partial }{%
\partial x}X_{s-u}^{x}\}^{\ast }dsdu]^{\ast }dx.
\end{eqnarray*}%
Then Lemma \ref{WeakConvergence} and dominated convergence combined with the
estimate (\ref{BoundedDerivatives}) give%
\begin{equation*}
\left\Vert I_{3}(n)\right\Vert \underset{n\longrightarrow \infty }{%
\longrightarrow }0.
\end{equation*}%
We also have that%
\begin{equation*}
-\int_{U}E[\Phi (X_{T}^{x,n})]\frac{\partial }{\partial x}\varphi (x)dx%
\underset{n\longrightarrow \infty }{\longrightarrow }\int_{U}E[\Phi
(X_{T}^{x})]\frac{\partial }{\partial x}\varphi (x)dx.
\end{equation*}%
Hence,%
\begin{eqnarray*}
&&-\int_{U}E[\Phi (X_{T}^{x})]\frac{\partial }{\partial x}\varphi (x)dx \\
&=&C\int_{U}\varphi (x)E[\Phi (X_{T}^{x}) \\
&&\times \int_{0}^{T}u^{-H-\frac{1}{2}}\int_{u}^{T}a(s-u)(s-u)^{\frac{1}{2}%
-H}s^{H-\frac{1}{2}}\left( \frac{\partial }{\partial x}X_{s-u}^{x}\right)
^{\ast }dB_{s}ds]^{\ast }dx
\end{eqnarray*}

Finally, using the monotone class theorem combined with dominated
convergence and the Cauchy-Schwarz inequality, we can show the above
relation for Borel measurable functions $\Phi :\mathbb{R}^{d}\longrightarrow 
\mathbb{R}$ such that 
\begin{equation*}
\Phi (X_{T}^{\cdot })\in L^{2}(\Omega \times U,\mu \times dx)\text{.}
\end{equation*}%
So the proof follows.
\end{proof}

\begin{lemma}
\label{L. 2.12}\bigskip Let $b\in C_{b}^{1}(\left[ 0,T\right] \times \mathbb{%
R}^{d};\mathbb{R}^{d})$ and $\left\Vert b\right\Vert _{\infty }\leq 1$.
Suppose that $H<\frac{1}{6}$. Then there exists a $C<\infty $ and a
sufficiently small $\alpha >0$ (which depend on $H,d,T$, but not on $b$)
such that for $0\leq s<t\leq T$:%
\begin{equation*}
E\left[ \exp (\frac{\alpha }{\left\vert t-s\right\vert ^{2(1-3H)}}\left\Vert
\int_{s}^{t}D_{x}b(u,B_{u}^{H})du\right\Vert ^{2})\right] <C,
\end{equation*}%
where $D_{x}$ is the Fr\'{e}chet derivative of $b$ with respect to the
spatial variable $x$.
\end{lemma}

\begin{remark}
\bigskip In fact, the proof of Lemma 2.12 in (Amine, Mansouri, Proske) gives
the following bound%
\begin{eqnarray}
&&E\left[ \exp (\frac{\alpha }{\left\vert t-s\right\vert ^{2(1-3H)}}%
\left\Vert \int_{s}^{t}D_{x}b(u,B_{u}^{H})du\right\Vert ^{2})\right]  \notag
\\
&\leq &E\left[ \exp (\alpha C(H,d,T)\left\Vert b\right\Vert _{\infty
}^{2}(1+\sup_{0\leq l\leq T}\left\vert B_{l}\right\vert )^{2})\right]
<\infty .  \label{Lemma 2.12}
\end{eqnarray}%
for a constant $C(H,d,T)<\infty $.
\end{remark}

\bigskip We also want to apply the following special version of a lemma,
which is due to Garcia, Rodemich and Rumsey (see e.g. \cite%
{GarciaRodemichRumsey})

\begin{lemma}
\label{GarciaRodemichRumsey}Let $\Lambda $ be a compact interval endowed
with a metric $d$. Define $\sigma (r)=\inf_{x\in \Lambda }\lambda (B(x,r))$,
where $B(x,r):=\left\{ y\in \Lambda :d(x,y)\leq r\right\} $ denotes the ball
of radius $r$ centered in $x\in \Lambda $ and where $\lambda $ is the
Lebesgue measure. Assume that $\Psi :\left[ 0,\infty \right) \longrightarrow %
\left[ 0,\infty \right) $ is positive, increasing and convex with $\Psi
(0)=0 $ and denote by $\Psi ^{-1}$ its inverse, which is a positive,
increasing and concave function. Suppose that $f:$ $\left[ 0,\infty \right)
\longrightarrow \left[ 0,\infty \right) $ is continuous on $\left( \Lambda
,d\right) $ and let%
\begin{equation*}
U=\int_{\Lambda \times \Lambda }\Psi \left( \frac{\left\vert
f(t)-f(s)\right\vert }{d(t,s)}\right) dtds\text{.}
\end{equation*}%
Then%
\begin{equation*}
\left\vert f(t)-f(s)\right\vert \leq 18\int_{0}^{d(t,s)/2}\Psi ^{-1}\left( 
\frac{U}{(\sigma (r))^{2}}\right) dr.
\end{equation*}
\end{lemma}
\paragraph*{}
\begin{align*}
&\textit{Emmanuel Coffie},&&\\
&\textit{Department of Mathematics and Statistics},&&\\ 
&\textit{University of Strathclyde},&&\\
&\textit{Glasgow, G1 1XH, UK}, &&\\
&\textit{emmanuel.coffie@strath.ac.uk}.&&\\\\
&\textit{Sindre Duedahl},&&\\
&\textit{Danske Bank, N-0250},&&\\
&\textit{Aker Brygge, Oslo, Norway},&&\\ 
&\textit{sidu@danskebank.com}.&&\\\\
&\textit{Frank Proske},&&\\
&\textit{Department of Mathematics},&&\\
&\textit{University of Oslo, N-0316},&&\\
&\textit{Blindern, Oslo, Norway},&&\\
&\textit{proske@math.uio.no}.&&
\end{align*}
\paragraph*{}

\end{document}